\documentclass[a4paper,11pt]{amsart}
\usepackage{amsfonts}
\usepackage{amssymb}
\usepackage[utf8]{inputenc}
\usepackage{amsmath}
\usepackage{pdflscape}
\usepackage{graphicx}
\usepackage[colorlinks=true, linkcolor=red, citecolor=blue]{hyperref}
\usepackage[]{epsfig}
\usepackage[]{pstricks}
\usepackage{tikz}

\newtheorem{theorem}{Theorem}[section]
\newtheorem{proposition}[theorem]{Proposition}
\newtheorem{lemma}[theorem]{Lemma}

\newtheorem{assumption}[theorem]{Assumption}

\theoremstyle{remark}

\usepackage{mathrsfs}
\def\supp{\operatorname{{supp}}}

\numberwithin{equation}{section}
\makeatletter
\@namedef{subjclassname@2010}{\textup{2020} Mathematics Subject Classification}
\makeatother

\begin{document}

	\pagenumbering{arabic}	
\title[Delayed KdV-type system]{Stabilization results for delayed fifth order KdV-type equation in a bounded domain}
\author[Capistrano-Filho]{Roberto de A. Capistrano-Filho}
\address{Departamento de Matem\'atica,  Universidade Federal de Pernambuco (UFPE), 50740-545, Recife (PE), Brazil.}
\email{roberto.capistranofilho@ufpe.br}
\author[Gonzalez Martinez]{Victor H. Gonzalez Martinez*}
\email{victor.hugo.gonzalez.martinez@gmail.com}
\thanks{*Corresponding author: victor.hugo.gonzalez.martinez@gmail.com}

\begin{abstract}Studied here is the Kawahara equation, a fifth order Korteweg-de Vries type equation, with time-delayed internal feedback. Under suitable assumptions on the time delay coefficients we prove that solutions of this system are exponentially stable.  First, considering a damping and delayed system,  with some restriction of the spatial length of the domain, we prove that the Kawahara system is exponentially stable for $T>  T_{\min}$.  After that,  introducing a more general delayed system, and by introducing suitable energies, we show using Lyapunov approach, that the energy of the Kawahara equation goes to zero exponentially,  considering the initial data small and a restriction in the spatial length of the domain. To remove these hypotheses, we use the \textit{compactness-uniqueness argument} which reduces our problem to prove an \textit{observability inequality}, showing a semi-global stabilization result.
\end{abstract}

\maketitle

\section{Introduction\label{intr}}

\subsection{Model under consideration} The full water wave system is too complex to allow to easily derive and rigorously from it relevant qualitative information on the dynamics of the waves. Alternatively, under suitable assumption on amplitude, wavelength, wave steepness and so on, the study on asymptotic models for water waves has been extensively investigated to understand the full water wave system, see, for instance, \cite{ASL,BCL, BLS, Lannes, Saut} and references therein for a rigorous justification of various asymptotic models for surface and internal waves.

Particularly, formulating the waves as a free boundary problem of the  incompressible, \linebreak irrotational Euler equation in an appropriate non-dimensional form, one has two non-dimensional parameters $\delta := \frac{h}{\lambda}$ and $\varepsilon := \frac{a}{h}$, where the water depth, the wave length and the amplitude of the free surface are parameterized as $h, \lambda$ and $a$, respectively. Moreover, another non-dimensional parameter $\mu$ is called the Bond number, which measures the importance of gravitational forces compared to surface tension forces. The physical condition $\delta \ll 1$ characterizes the waves, which are called long waves or shallow water waves, but there are several long wave approximations according to relations between $\varepsilon$ and $\delta$. For example,  when $\varepsilon = \delta^4 \ll 1$, $\mu = \frac13 + \nu\varepsilon^{\frac12}$ and, in connection with the critical Bond number, $\mu = \frac13$, Hasimoto \cite{Hasimoto1970} derived a fifth-order KdV equation of the form
\[\pm2 u_t + 3uu_x - \nu u_{xxx} +\frac{1}{45}u_{xxxxx} = 0,\]
which is nowadays called the Kawahara equation.

Our main focus in this work is to investigate the behavior of the solution of the Kawahara equation \cite{Hasimoto1970,Kawahara}, a fifth higher-order Korteweg-de Vires (KdV) equation
\begin{equation}\label{fda1}
u_{t}+u_{x}+u_{xxx}-u_{xxxxx}+uu_{x}=0
\end{equation}
which is a dispersive PDE describing numerous wave phenomena such as
magneto-acoustic waves in a cold plasma \cite{Kakutani}, the propagation of
long waves in a shallow liquid beneath an ice sheet \cite{Iguchi}, gravity
waves on the surface of a heavy liquid \cite{Cui}, etc. In the literature this
equation is also referred as the fifth-order KdV equation \cite{Boyd}, or
singularly perturbed KdV equation \cite{Pomeau}.

There are some valuable efforts in the last years that focus on the analytical
and numerical methods for solving \eqref{fda1}. These methods
include the tanh-function method \cite{Berloff}, extended tanh-function method
\cite{Biswas}, sine-cosine method \cite{Yusufoglu}, Jacobi elliptic functions
method \cite{Hunter}, direct algebraic method \cite{Polat}, decompositions
methods \cite{Kaya}, as well as the variational iterations and homotopy
perturbations methods \cite{Jin}. 
For numerical simulations, however, there
appears the question of cutting-off the spatial domain \cite{Bona1, Bona2}.
This motivates the detail qualitative analysis of problems for \eqref{fda1}
in bounded regions.

\subsection{Setting of the problem and main results} Our goal in this manuscript is to analyze qualitative properties of solutions to the
initial-boundary value problem for \eqref{fda1} posed on a bounded interval
under the presence of a localized damping and delay terms, that is
\begin{equation}
\begin{cases}\label{fd1}
u_{t}(x,t)+u_{x}(x,t)+u_{xxx}(x,t)-u_{xxxxx}(x,t)+u(x,t)u_{x}(x,t)\\ \quad  \quad  \quad  \quad  \quad  \quad  \quad  \quad  \quad  \quad \quad  +a\left(x\right)u(x,t) +b(x)u(x,t-h)=0 & x \in (0,L),~ t>0,\\
u\left(  0,t\right)  =u\left(  L,t\right)  =u_{x}\left(  0,t\right)
=u_{x}\left(  L,t\right)  =u_{xx}\left(  L,t\right)  =0 & t>0,\\
u\left(  x,0\right)  =u_{0}\left(  x\right)  & x \in (0,L),\\
u(x,t)=z_0(x,t) & x \in (0,L),~t \in (-h,0),
\end{cases}
\end{equation}
where $h>0$ is the time delay,  $L>0$ is the length of the spatial domain,  $u(x,t)$ is the amplitude of the water wave at position $x$ at time $t$, and $a = a(x)$ and $b = b(x)$ are nonnegative functions belonging to $L^\infty(\Omega)$. For our purpose let us introduce the following assumption.
\begin{assumption}\label{A1}
	The real functions $a=a\left(x\right)$, $b=b\left(x\right)$ are nonnegative functions belonging to $L^\infty(\Omega)$. Moreover, $a(x) \geq a_0>0$ almost everywhere in a nonempty subset $\omega \subset (0,L)$.
\end{assumption}

Note that the term $a(x)u$ designs a feedback damping mechanism (see, for instance \cite{CaArDo}); therefore, one can expect the global well-posedness of \eqref{fd1} for all $L>0$, and the decay of solutions.  Thus, defining the energy of system \eqref{fd1} by
	\begin{equation}\label{Eb}
	E_u(t)=\frac{1}{2}\int_0^L u^2(x,t)dx+\frac{h}{2}\int_0^L\int_0^1 b(x)u^2(x,t-\rho h) d \rho dx,
	\end{equation}
the following questions arise: 
	\begin{itemize}
	\item[i.]Does $E_u(t)\longrightarrow0$, as $t\to\infty$?
	\item[ii.] If it is the case, can we give the decay rate?
	\end{itemize}
So, the main purpose of this paper is to answer these questions. There are basically three features to be emphasized in this way.
\begin{itemize}
\item The damping is effectively important, i.e.
there are solutions to undamped model (at least to its linear version) that do
not decay \cite{CaArDo};

\item The nonlinear term can be estimated in appropriate
norms, i.e. there are suitable functional spaces that allow to apply
corresponding methods;

\item The delay in the feedback does not destabilize the
system, which can be the case for other delayed systems, see for instance \cite{Datko, Nicaise2006, Valein}.
\end{itemize}

Our first result ensures that with a restrictive assumption on the length $L$ of the domain and with the weight of the delayed feedback small enough the solutions of the system  \eqref{fd1} are locally stable. 

\begin{theorem}\label{T1.2J}
Assume that $\supp b(x) \not\subset \supp a(x)$ and the functions satisfying the conditions given in Assumption \ref{A1}. Let $L<\pi \sqrt{3}$, $\xi >1$ and $T_0$ given by
\begin{equation}\label{To}
T_0=\frac{1}{2\gamma}\ln\left( \frac{2\xi \kappa }{\eta}\right)+1.
\end{equation}
Thus, for $s \in [0,T_0)$ with
\begin{equation}\label{Tmin}
T_{\min}:=-\frac{1}{\nu}\ln\left(\frac{\eta}{2}\right)+\left(\frac{2\|b\|_\infty}{\nu}+1\right)s,
\end{equation}
there exists $\delta>0$ (depending on $\xi, L, h$) and $r>0$ sufficiently small such that if $\|b\|_\infty<\delta$, for every $(u_0, z_0) \in\mathcal{H}=L^2(0,L) \times L^2((0,L)\times(0,1))$ satisfying $\|(u_0,z_0)\|_{ \mathcal{H}} \leq r$, the energy \eqref{Eb} of the system \eqref{fd1} decays exponentially for $t>T_{\min}$.
\end{theorem}

Another goal of this paper,  inspired  by the work of Nicaise and Pignotti \cite{Nicaise2006}, is to consider the following system
\begin{equation}
	\begin{cases}\label{J1}
		u_{t}(x,t)+u_{x}(x,t)+u_{xxx}(x,t)-u_{xxxxx}(x,t)+u(x,t)u_{x}(x,t)\\ \quad  \quad  \quad  \quad  \quad  \quad  \quad  \quad  \quad  \quad \quad +a\left(x\right)(\mu_1 u(x,t)+\mu_2u(x,t-h))=0 & x \in (0,L),~ t>0,\\
		u\left(  0,t\right)  =u\left(  L,t\right)  =u_{x}\left(  0,t\right)
		=u_{x}\left(  L,t\right)  =u_{xx}\left(  L,t\right)  =0 & t>0,\\
		u\left(  x,0\right)  =u_{0}\left(  x\right)  & x \in (0,L),\\
		u(x,t)=z_0(x,t) & x \in (0,L),~t \in (-h,0),
	\end{cases}
\end{equation}
called now on of $\mu_i-$system. Here $h>0$ is the time delay,  $\mu_1>\mu_2$ are positive real number and the initial data $(u_0, z_0)$ belong to a suitable space.  If $a(x)$ satisfies Assumption \ref{A1}, consider the following energy associated to the solutions of the system \eqref{J1}
\begin{equation}\label{J6}
E_u(t)=\frac{1}{2}\int_0^Lu^2(x,t)dx+\frac{\xi}{2}\int_0^L\int_0^1 a(x)u^2(x,t-\rho h) d\rho dx,
\end{equation}
where $\xi$ is a positive constant verifying the following
\begin{equation}\label{cdelay}
	h\mu_2<\xi<h(2\mu_1-\mu_2).
\end{equation}
Again, we are interested to see the questions previously mentioned.  Note that, in a different way of our first goal, the derivative of the energy \eqref{J6} satisfies 
$$
	E_u'(t)\leq-C\left[u_{xx}^2(0)+\int_{0}^{L}a(x)u^2(x)dx+\int_{0}^{L}a(x)u^2(x,t-h)dx\right],
$$
for some positive constant $C:=C(\mu_1,\mu_2,\xi,h)$. This indicates that the function $a(x)$ plays the role of a feedback damping mechanism, at least for the linearized system. Therefore, questions previously mentioned again arise to the solution of the system \eqref{J1}.
	\begin{itemize}
	\item[i.]Does $E_u(t)\longrightarrow0$, as $t\to\infty$?
	\item[ii.] If it is the case, can we give the decay rate?
	\end{itemize}

With all this information in hand,  the following two results are related to the system \eqref{J1}.  For this system we split the result of the behavior in two parts, the first one, by using Lyapunov approach, gives that the energy \eqref{J6} tends to zero, however, if the initial data are small enough. Precisely,  the local result can be read as follows.

\begin{theorem}\label{P6J}
Let $L>0$, assume that $a \in L^\infty(\Omega)$, \eqref{cdelay} holds and $L<\pi \sqrt{3}$. Then, there exists $0<r<\frac{9\pi^2-3L^2}{2L^{\frac{3}{2}}\pi^2}$ such that for every $(u_0,z_0(\cdot,-h(\cdot))) \in \mathcal{H}$ satisfying $\|(u_0,z_0(\cdot,-h(\cdot)))\|_{\mathcal{H}} \leq r,$ the energy \eqref{J6} of the system \eqref{J1} decays exponentially. More precisely, there exist two positive constants $\gamma$ and $\kappa$ such that
\begin{equation*}
	E(t)\leq \kappa E(0) e^{-2\gamma t} \hbox{ for all }  t>0.
\end{equation*}
Here, 
\begin{align*}
\gamma \leq & \min\left\{\frac{9\pi^2-3L^2-2L^{\frac{3}{2}} r\pi^2}{3L^2(1+2L\alpha)}\alpha,\frac{\beta \xi}{2h(\xi \beta+\xi)} \right\},\\
	\kappa = & \left(1+\max\{2\alpha L,\beta\}\right),
\end{align*}
with $\alpha$ and $\beta$ positive constants such that
\begin{align*}
\alpha<& \min\left\{\frac{1}{2L\mu_1+L\mu_2}\left(\mu_1-\frac{\xi}{2h}-\frac{\mu_2}{2}-\frac{\beta \xi}{2h}\right),\frac{1}{L\mu_2}\left(\frac{\xi}{2h}-\frac{\mu_2}{2}\right) \right\},\\
\beta<& \frac{2h}{\xi} \left(\mu_1-\frac{\xi}{2h}-\frac{\mu_2}{2}\right).
\end{align*}
\end{theorem}

The last result of the manuscript, still related with the system \eqref{J1}, removes the hypothesis of the initial data being small.  To do that, we use compactness-uniqueness argument due to J.-L. Lions \cite{Lions}, which reduces our problem to prove an \textit{observability inequality} for the nonlinear system \eqref{J1}.  Thus,  we have the following semi-global result.

\begin{theorem}\label{T1.1J}
	Assume that $a(x)$ satisfying the Assumption \ref{A1}. Suppose that $\mu_1 > \mu_2$ and let $\xi>0$ satisfying \eqref{cdelay}. Let $R>0$, then there exists $C=C(R)>0$ and $\nu=\nu(R)>0$ such that $E_u$, defined in \eqref{J6}, satisfies 
	\begin{equation*}
		E_u(t)\leq CE_u(0)e^{-\nu t}, \quad \forall t>0,
	\end{equation*}
for solutions of \eqref{J1} provided that $\|(u_0,z_0)\|_{\mathcal{H}} \leq R$.
\end{theorem}

\subsection{Previous results} Let us now mention some bibliography comments about the stabilization problem for KdV-type models.  With respect to the Kawahara equation with damping term recently,  in \cite{CaArDo}, the authors considered this system
\begin{equation}\label{fda1a}
u_{t}+u_{x}+u_{xxx}-u_{xxxxx}+u^pu_{x}+a(x)u=0,\quad (x,t)\in(0,L)\times(0,T),
\end{equation}
for $p\in[1,4)$, with a presence of an extra damping term $a(x)$, such that 
\begin{equation}\label{fd2}
\left\{\begin{array}{l}
a \in L^{\infty}(0, L) \text { and } a(x) \geqslant a_{0}>0 \quad \text { a.e. in } \omega \\
\text { with a nonempty } \omega \subset(0, L).
\end{array}\right.
\end{equation}
This damping mechanism is essential already in a linear case: if $a(x)\equiv0,$ then a nontrivial solution to
\begin{equation*}
	\begin{cases}
	u_{t}+u_{x}+u_{xxx}-u_{xxxxx}=0,&(x,t)\in (0,L) \times (0,T)\\
	u\left(  0,t\right)  =u\left(  L,t\right)  =u_{x}\left(  0,t\right)
=u_{x}\left(  L,t\right)  =u_{xx}\left(  L,t\right)  =0,&t\in\left(  0,T\right)\\
u\left(  x,0\right)  =u_{0}\left(  x\right), &x\in(0,L),
	\end{cases}
	\end{equation*}
is constructed to be not decayed as $t\rightarrow\infty$ if the length of an
interval is critical (see \cite{CaArDo}). Observe that due to the drift term $u_{x}$ the same
occurs for the KdV equation \cite{Rosier}. Indeed, if for instance $L=2\pi
n,\ n\in\mathbb{N},$ then the function $v(x)=1-\cos x$ solves
\begin{equation*}
	\begin{cases}
	u_{t}+u_{x}+u_{xxx}=0,&(x,t)\in (0,L) \times (0,T)\\
	u\left(  0,t\right)  =u\left(  L,t\right) =u_{x}\left(  L,t\right)  =0,&t\in\left(  0,T\right)\\
u\left(  x,0\right)  =u_{0}\left(  x\right), &x\in(0,L),
	\end{cases}
	\end{equation*}
and clearly $v(x)\not \rightarrow 0$ as $t\rightarrow\infty.$ Despite the
valuable advances in \cite{Cerpa, Cerpa1, Coron, Goubet}, the question whether
solutions of undamped problems associated to nonlinear KdV and Kawahara equations decay as $t\to\infty$ for all finite $L>0$ is still open.

To overcome these difficulties, a damping of the type $a(x)u$ was introduced
in \cite{Zuazua1} to stabilize the KdV system. More precisely, considering the
damping localized at a subset $\omega\subset\left(  0,L\right)  $ containing
nonempty neighborhoods of the end-points of an interval, it was shown that
solutions of both linear and nonlinear problems for the KdV equation decay,
independently on $L>0.$ In \cite{Pazoto} it was proved that the same holds
without cumbersome restrictions on $\omega\subset(0,L).$ In
\cite{vasconcellos, vasc1} the damping like in \eqref{fd2} was used for
\eqref{fda1a} without the drift term $u_{x}.$ If, however, the linear term
$u_{x}$ is dropped, both the KdV and Kawahara equations do not possess
critical set restrictions \cite{Rosier, vasc2}, and the damping is not
necessary. The decay of solutions in such case was also proved in
\cite{Doronin2, Faminskii} by different methods.

Once the damping term $a(x)u\not \equiv 0$ is added to \eqref{fda1a}, the
nonlinearity $uu_{x}$ provides the second difficulty which should be
treated with accurateness. In this context the mixed problems for the
generalized KdV equation
\begin{equation}
u_{t}+u_{x}+u_{xxx}+u^{p}u_{x}+a\left(  x\right)  u=0, \label{kdv1}%
\end{equation}
were studied in \cite{Rosier1} when $p\in[2,4).$ For the critical exponent,
$p=4,$ the global well-posedness and the exponential stability were studied in
\cite{Lin-Paz}. The reader is also referred to \cite{Komornik, RussellZhang}
and the references therein for an overall literature review. 

Still related with damping mechanism for dispersive model,  more recently,  Cavalcanti \textit{et al.}  \cite{CCKR} studied a damped KdV--Burgers equation in the real line,
\begin{equation}\label{14}
\begin{cases}
u_t(x,t)+u_{xxx}(x,t)-u_{xx}(x,t) +\lambda_0 u(x,t) +u(x,t) u_x(x,t)=0
&(x,t)\in\mathbb{R}\times (0,\infty),\\
u(x, 0)=u_0(x)&x\in\mathbb{R}
\end{cases}
\end{equation}
The authors were able to show the well-posedness and exponential stability for an indefinite damping $\lambda_0(x),$ giving exponential decay estimates on the $L^2-$norm of solutions to \eqref{14} under appropriate conditions on the damping coefficient $\lambda_0.$ Additionally,  recently a work due to Komornik and Pignotti \cite{KomPig2020} studied the following equation
\begin{equation}\label{11}
\begin{cases}
u_t(x,t)+u_{xxx}(x,t)-u_{xx}(x,t) +\lambda_0 u(x,t)\\
\hspace{4cm}+\lambda u(x,t-\tau )+u(x,t) u_x(x,t)=0
&(x,t)\in\mathbb{R}\times (0,\infty),\\
u(x, s)=u_0(x, s)&x\in\mathbb{R}\times [-\tau, 0].
\end{cases}
\end{equation}
Precisely,  the authors consider the system \eqref{11} in presence of a damping term and delay feedback.  They proved the exponential decay estimates under appropriate conditions on the damping \linebreak coefficients. 

It is important to point out that very recently, the robustness with respect to the delay of the boundary stability of the nonlinear KdV equation has been studied in \cite{BaCrVa}. The authors obtain, under an appropriate condition on the feedback gains with and without delay  the locally exponentially stable result for non critical length.  Moreover,  in \cite{Valein}, the authors extend this result for the nonlinear Korteweg-de Vries equation in the presence of an internal delayed term. This work is our motivation to treat more general dispersive systems in this manuscript.

\subsection{Heuristic of the article}
In this article, we investigated and discussed the stabilization problems with the damping mechanism and feedback delay of a fifth order KdV, known as the Kawahara equation. As we can see in this introduction, the agenda of the research of stabilization problems for the Kawahara equation is quite new and does not acknowledge many results in the literature. With this proposal to fill this gap, we intend to introduce two ways to treat the stabilization of the solution for the Kawahara equation with damping and delay terms. In both cases, under suitable assumptions, we prove exponential stability results of the solution which are obtained by introducing suitable energies, using Lyapunov approach and compactness-uniqueness argument.

First, the strategy to treat Theorem \ref{T1.2J} is the following: We first prove the exponential stability for the Kawahara system linearized around $0$ (see Appendix \ref{ApA}) by the Lyapunov approach for all $L< \pi \sqrt{3}$ (allowing to have an estimation of the decay rate), then for $\|b\|_\infty$ small enough, we show the local exponential stability result by a decoupling approach inspired in \cite{Valein}.

The second result or the manuscript, Theorem \ref{P6J}, has a local character, that is, is necessary to make the initial data small enough.  The local stability result is based on the appropriate choice of Lyapunov functional, which one gives a restriction of the lengths $L$. This happens by the fact from the choice of the Morawetz multipliers $x$ in the expression of $V_1$ (see \eqref{J27}).

Finally,  Theorem \ref{T1.1J}, is based to prove an observability inequality, for the nonlinear delayed Kawahara equation which one is proved using a contradiction argument.  Consequently, the value of the decay rate can not be estimated in this approach, differently than before. The two main difficulties to the semi-global stability result are the pass to the limit in the nonlinear term and the fact that this nonlinear term does not allow to use Holmegren's uniqueness theorem. Instead we will use the unique continuation property for the nonlinear system due to Saut and Scheurer \cite{Saut}.  In this case,  the results follows without restriction in the length $L>0$.

\vspace{0.1cm}

We finish our introduction with a few comments that give a generality of the problems in consideration.
\begin{itemize}
\item First, observe that to prove Theorem \ref{T1.2J} we do not need to localize the solution of the transport equation\footnote{See the equation \eqref{J10} below.} in a small subset of $(0,L)$ as in \cite[Section 4]{Valein}. Moreover, we emphasize that we can take $a=0$ in the Theorem \ref{T1.2J}. Finally, it is important to mention that we do not know if the time $T_{\min}$ is optimal. 
\vspace{0.1cm}
\item It is important to point out that the Theorem \ref{P6J} gives an estimation of the decay rate $\gamma$. In particular, we can note that when the delay $h$ increases, the decay rate $\gamma$ decreases. 
\vspace{0.1cm}
\item Note that in Theorems \ref{P6J} and \ref{T1.1J} the relation \eqref{cdelay} is more general that one used in \cite{Valein}.  Our motivation is the general framework introduced by Nicaise and Pignotti in \cite{Nicaise2006}.
\vspace{0.1cm}
\item As mentioned before, Theorem  \ref{T1.1J} has a semi-global character. This comes from the fact that even if we are able to choose any radius $R$ for the initial data, the decay rate $\nu$ (see \eqref{nu_a}) depends on $R$.
\vspace{0.1cm}
\item The previous results are not only true for the nonlinearity $uu_x$. Using the same approach as in \cite{CaArDo}, we can deal with a general nonlinearity as $u^{p}u_{x}$ for  $p\in[1,4)$.  For simplicity here we will treat the case $p=1$.
\vspace{0.1cm}
\item As mentioned in the beginning of the introduction, there are several long wave \linebreak approximations according to relations between $\varepsilon$ and $\delta$, for instance, to the KdV equation \cite{Korteweg}\footnote{This equation was firstly introduced by Boussinesq \cite{Boussinesq}, and Korteweg and de Vries rediscovered it twenty years later.} we can take $\varepsilon = \delta^2 \ll 1$ and $\mu \neq \frac13$ which give us
\[\pm2 u_t + 3uu_x +\left( \frac13 - \mu\right)u_{xxx} = 0.\]
Connecting the KdV and Kawahara equations, in \cite{LucWan} the authors studied limit behavior of the solutions of the Kawahara equation $$ u_{t}+u_{xxx}+\varepsilon u_{xxxxx}+u u_{x}=0, \quad \varepsilon>0 $$ as $\varepsilon \rightarrow 0 .$ Note that in this previous equation  $u_{xxx}$ and $\varepsilon u_{xxxxx}$ compete each other and cancel each other at frequencies of order $1 / \sqrt{\varepsilon}$.  Thus, the authors proved that the solutions to this equation converge in $C\left([0, T] ; H^{1}(\mathbb{R})\right)$ towards the solutions of the KdV equation for any fixed $T>0$. Due to this previous fact, we believe that considering an approximation of the delayed system in the bounded domain
\begin{equation*}
\begin{cases}
u_{t}(x,t)+u_{x}(x,t)+u_{xxx}(x,t)+\varepsilon u_{xxxxx}(x,t)+u(x,t)u_{x}(x,t)\\ \quad  \quad  \quad  \quad  \quad  \quad  \quad  \quad  \quad  \quad \quad \quad  +a\left(x\right)u(x,t)  +b(x)u(x,t-h)=0 & x \in (0,L),~ t>0,\\
u\left(  x,0\right)  =u_{0}\left(  x\right)  & x \in (0,L),\\
u(x,t)=z_0(x,t) & x \in (0,L),~t \in (-h,0),
\end{cases}
\end{equation*}
with the compatible $\varepsilon-$boundary condition,  using the approach of our work, we can recover (as $\varepsilon \rightarrow 0$) the results proposed by \cite{Valein}.
\end{itemize}

\subsection{Organization of the article} Our manuscript is outlined as follows: First, Section \ref{Sec2} is related with the well-posedness results for $\mu_i-$system \eqref{J1} and its adjoint.  After that, Section \ref{Sec3} is devoted to prove properties of the damping--delayed system \eqref{fd1},  that is,  we show the Theorem \ref{T1.2J}, where the analysis developed in the Appendix \ref{ApA} is crucial.  In Section \ref{Sec4}, we give a rigorous proof of the asymptotic stability for the solutions of the system \eqref{J1}, precisely, we prove Theorem \ref{P6J}. After that, in this same section, to remove restrictions of the Theorem \ref{P6J}, we prove an \textit{observability inequality},  which is the key to prove Theorem \ref{T1.1J}.

\section{Well-posedness of  $\mu_i-$system \label{Sec2}}
Our goal in this section is to prove the well-posedness theory for the system \eqref{J1}. This analysis is useful for the stability properties for the solutions of this system.

\subsection{Linear system\label{S1}}
For the sake of completeness, we provide below the well-posedness results for
the linear system
\begin{equation}
	\begin{cases}\label{KW3}
		u_{t}(x,t)+u_{x}(x,t)+u_{xxx}(x,t)-u_{xxxxx}(x,t)\\ \quad  \quad  \quad  \quad  \quad  \quad  \quad  \quad  \quad  \quad \quad +a\left(x\right)(\mu_1 u(x,t)+\mu_2u(x,t-h))=0 & x \in (0,L),~ t>0,\\
		u\left(  0,t\right)  =u\left(  L,t\right)  =u_{x}\left(  0,t\right)
		=u_{x}\left(  L,t\right)  =u_{xx}\left(  L,t\right)  =0 & t>0,\\
		u\left(  x,0\right)  =u_{0}\left(  x\right)  & x \in (0,L),\\
		u(x,t)=z_0(x,t) & x \in (0,L),~t \in (-h,0).
	\end{cases}
\end{equation}

A classical way to deal with the well-posedness of the delayed equations (see e.g. \cite{Nicaise2006}) is to consider $z(x,\rho,t)=u(x,t-\rho h)$, for any $x \in (0,L)$, $\rho \in (0,1)$ and $t>0$.  So, its easily verified that $z$ satisfies the transport equation
\begin{equation}\label{T1}
	\begin{cases}
h z_t(x,\rho,t)+z_\rho(x,\rho,t)=0 & x \in (0,L), ~\rho \in (0,1), ~ t>0,\\
z(x,0,t)=u(x,t) & x \in (0,L),~ t>0,\\
z(x,\rho,0)=z_0(x,-\rho h) & x \in (0,L), ~\rho \in (0,1).
	\end{cases}
\end{equation}
We equipped the Hilbert space $\mathcal{H}=L^2(0,L) \times L^2((0,L)\times(0,1))$ with the following inner product 
\begin{equation*}
	((u,z), (v,w))_{\mathcal{H}} = \int_0^L uv dx+\xi \|a\|_\infty \int_{0}^{L}\int_0^1 z(x,\rho)w(x,\rho) dxd\rho,
\end{equation*}
where $\xi$ is a positive constant satisfying \eqref{cdelay} or, equivalently, 
\begin{equation}\label{2.7}
	\mu_2<\frac{\xi}{h}<2\mu_1-\mu_2.
\end{equation}
that we will use from now on.

To study the well-posedness theory in the sense of Hadamard,  we need to put the equation \eqref{KW3} into an abstract setting. To do it, let us rewrite this system as follows: Consider $U(t)=(u,z(\cdot,\cdot,t))$,  so the equation \eqref{KW3} can be reformulated as the following system
\begin{equation}
\begin{cases}\label{J9}
	u_{t}(x,t)+u_{x}(x,t)+u_{xxx}(x,t)-u_{xxxxx}(x,t)\\ \quad  \quad  \quad  \quad  \quad  \quad  \quad  \quad  \quad  \quad \quad +a\left(x\right)(\mu_1 u(x,t)+\mu_2z(1))=0 & x \in (0,L),~ t>0,\\
	u\left(  0,t\right)  =u\left(  L,t\right)  =u_{x}\left(  0,t\right)
	=u_{x}\left(  L,t\right)  =u_{xx}\left(  L,t\right)  =0 & t>0,\\
	u\left(  x,0\right)  =u_{0}\left(  x\right)  & x \in (0,L),\\
	h z_t(x,\rho,t)+z_\rho(x,\rho,t)=0 & x \in (0,L), ~\rho \in (0,1), ~ t>0,\\
	z(x,0,t)=u(x,t) & x \in (0,L),~ t>0,\\
	z(x,\rho,0)=z_0(x,-\rho h) & x \in (0,L), ~\rho \in (0,1),
\end{cases}
\end{equation}
which is equivalent to the following abstract Cauchy problem
\begin{equation}\label{J10}
	\begin{cases}
		\displaystyle \frac{\partial U}{\partial t}(t)=\mathcal{A}U(t),\\
		U(0)={}  (u_{0}(x), z_0(x,-\rho h )).
	\end{cases}
\end{equation}
Here,  the unbounded operator $\mathcal{A}:\mathscr{D}(\mathcal{A}) \subset \mathcal{H} \rightarrow \mathcal{H}$ is given by
\begin{equation}\label{A}
	\mathcal{A}(u,z)=(-u_x-u_{xxx}+u_{xxxxx}-a(x)(\mu_1 u+\mu_2z(\cdot,1)),-h^{-1}z_\rho)
\end{equation}
with domain
\begin{equation}\label{2.13}
	\mathscr{D}(\mathcal{A})  = \left\{(u,z) \in \mathcal{H}: 
	\begin{array}{c}
	\displaystyle	u \in H^5(0,L), u(0)=u(L)=u_x(0)=u_x(L)=u_{xx}(L)=0,\\
	\displaystyle  z_\rho \in L^2((0,L)\times (0,1)),	z(0)=u
	\end{array}
	\right\}. 
\end{equation}
Let us now denote $z(1):=z(x,1,t)$ in system \eqref{J9} which will be used throughout the article. Thus, the first result of this section gives some properties of the operator $A$ and its adjoint $A^{*}$.
\begin{lemma}
	The operator $\mathcal{A}$ is closed and its adjoint $\mathcal{A}^{*}:\mathscr{D}(\mathcal{A}^*)\subset \mathcal{H} \rightarrow \mathcal{H}$ is given by
	\begin{equation}\label{Astar}
		\mathcal{A}^*(u,z)=(u_x+u_{xxx}-u_{xxxxx}-a(x)\mu_1 u+\frac{\xi \|a\|_\infty}{h}z(\cdot,0),h^{-1}z_\rho)
	\end{equation}
	with domain
	\begin{equation}\label{2.14}
		\mathscr{D}(\mathcal{A}^*)  = \left\{(u,z) \in \mathcal{H}: 
		\begin{array}{c} \displaystyle
			u \in H^5(0,L),  u(0)=u(L)=u_x(0)=u_x(L)=u_{xx}(0)=0,\\
			\displaystyle z_\rho \in L^2((0,L) \times (0,1)), z(x,1)=-\frac{a(x) h \mu_2}{\|a\|_\infty\xi}u(x)
		\end{array}
		\right\}. 
	\end{equation}
\end{lemma}
\begin{proof}
The proof that $A^*$ is given as in the statement of the lemma is standard. To show that $A$ is closed, note that $A^{**}=A$ and the result follows from \cite[Proposition 2.17]{Brezis}.
\end{proof}

Now, we are able to prove that $\mathcal{A}$ is the infinitesimal generator of a $C_0$-semigroup. Precisely, the result can be read as follows.
\begin{proposition}\label{Agi}
	Assume that $a \in L^\infty(\Omega)$ and \eqref{2.7} holds. Then, $\mathcal{A}$ is the infinitesimal generator of a $C_0$-semigroup in $\mathcal{H}$.
\end{proposition}
\begin{proof}
Let $U =(u,z) \in \mathscr{D}(\mathcal{A})$, then integrating by parts, using the definition of \eqref{2.13} and Young's inequality we have 
\begin{align*}
	(\mathcal{A}U,U)_{\mathcal{H}} \leq & -\frac{1}{2}u_{xx}^{2}(0)+\left(-\mu_1+\frac{\mu_2}{2}\right)\int_{0}^{L}a(x)u^2(x)dx+\frac{\xi\|a\|_\infty}{2h}\int_{0}^{L}u^2(x)dx\\
	&+\left(\frac{\mu_2}{2}-\frac{\xi}{2h}\right)\|a\|_\infty\int_{0}^{L}z^2(x,1)dx\\
	\leq & \frac{\xi\|a\|_\infty}{2h}\int_{0}^{L}u^2(x)dx.
\end{align*}
Hence, for $\lambda=\frac{\xi\|a\|_\infty}{2h}$ we have
\begin{equation*}
	((\mathcal{A}-\lambda I)U,U)_{\mathcal{H}}\leq 0.
\end{equation*}

Now, let $U =(u,z) \in \mathscr{D}(\mathcal{A}^*)$, then analogously as done previously, we get
\begin{align*}
	(\mathcal{A}^*U,U)_{\mathcal{H}} \leq &  
	  -\frac{1}{2}u_{xx}^2(L)+ \left(-\mu_1+\frac{\mu_2^2h}{2\xi}\right)\int_{0}^{L}a(x)u^2(x)dx+\frac{\xi \|a\|_\infty}{2h}\int_{0}^{L}u^2(x)dx.
\end{align*}
So, the following relation, 
\begin{equation*}
	\mu_2<\frac{\xi}{h} \Rightarrow 2 \mu_2<\frac{2\xi}{h} \Rightarrow \frac{2}{\mu_2}<\frac{2\xi}{h\mu_2^2} \Rightarrow \frac{\mu_2^2h}{2\xi}<\frac{\mu_2}{2}
\end{equation*}
yields that
\begin{align*}
	(\mathcal{A}^*U,U)_{\mathcal{H}} \leq & -\frac{1}{2}u_{xx}^2(L)+ \left(-\mu_1+\frac{\mu_2}{2}\right)\int_{0}^{L}a(x)u^2(x)dx+\frac{\xi \|a\|_\infty}{2h}\int_{0}^{L}u^2(x)dx\\ \leq & \ \frac{\xi \|a\|_\infty}{2h}\int_{0}^{L}u^2(x)dx.
\end{align*}
Hence,
\begin{equation*}
	((\mathcal{A}-\lambda I)^*U,U)_{\mathcal{H}} \leq 0,
\end{equation*}
for all $U \in \mathscr{D}(\mathcal{A}^*)$. Finally, since $\mathcal{A}-\lambda I$ is densely defined closed linear operator, and both $\mathcal{A}-\lambda I$ and $(\mathcal{A}-\lambda I)^*$ are dissipative, then $\mathcal{A}$ is the infinitesimal generator of a $C_0$-semigroup on $\mathcal{H}$ (see, for instance, \cite[Corollary 4.4]{Pazy} and \cite[Remark before Corollary 3.8]{Pazy}).
\end{proof}

The following theorem gives the existence of solutions for the abstract system \eqref{J10}.
\begin{theorem}
	Assume that $a \in L^\infty(\Omega)$ and \eqref{2.7} holds. Then, for each $U_0 \in \mathcal{H}$ there exists a unique mild solution $U \in C([0,\infty),\mathcal{H})$ for the system \eqref{J10}. Moreover, if $U_0 \in \mathscr{D}(\mathcal{A})$ the solutions are classical and satisfies the following regularity
	\begin{equation}\label{cCauhy}
	U \in C([0,\infty),\mathscr{D}(\mathcal{A}))\cap C^1([0,\infty),\mathcal{H}).
	\end{equation}
\end{theorem}
\begin{proof}
	The result is a direct consequence of Proposition \ref{Agi}.
\end{proof}

For $T>0$, $L>0$ let us introduce the following set
\begin{equation}\label{spaceB}
\mathcal{B}=C([0,T],L^2(0,L))\cap L^2(0,T,H_{0}^{2}(0,L))
\end{equation}
endowed with its natural norm
\begin{equation}\label{normB}
\|y\|_{\mathcal{B}}=\max_{t \in [0,T]}\|y(\cdot,t)\|_{L^2(0,L)}+\left( \int_{0}^{T}\|y(\cdot,t)\|_{H^2(0,L)}^2dt\right)^{\frac{1}{2}}.
\end{equation}
Next results are devoted to show \textit{a priori} and regularity estimates for the solutions of \eqref{J10}.

\begin{proposition}\label{P5JL}
Let $a \in L^\infty(\Omega)$ and consider that \eqref{2.7} holds. Then, for any mild solution of \eqref{J10} the energy $E_u$, defined by \eqref{J6}, is non-increasing and there exists a positive constant $C$ such that
\begin{equation}\label{L24L}
	E_u'(t)\leq-C\left[u_{xx}^2(0)+\int_{0}^{L}a(x)u^2(x)dx+\int_{0}^{L}a(x)u^2(x,t-h)dx\right]
\end{equation}
where $C$ is given by
\begin{equation}\label{C}
	C=\min\left\{\frac{1}{2},\mu_1-\frac{\xi}{2h}-\frac{\mu_2}{2},-\frac{\mu_2}{2}+\frac{\xi}{2h}\right\}.
\end{equation}
\end{proposition}
\begin{proof}
Multiplying $\eqref{J9}_1$ by $u(x,t)=z(x,0,t)$ and integrating over $(0,L)$ we infer that
\begin{equation}\label{2.45}
	\begin{gathered}
		\frac{1}{2}\frac{d}{dt} \| u(t)\|_{L^2(0,L)} =-\frac{1}{2}u_{xx}^{2}(0)- \mu_1\int_{0}^{L}a(x) u^2(x,t)dx-\mu_2\int_{0}^{L}a(x)u(x,t-h)u(x,t)dx.
	\end{gathered}
\end{equation}
Now, Multiplying $\eqref{J9}_4$ by $a(x)\xi u(x,t-\rho h)$ and integrating over $(0,L) \times (0,1)$ we obtain,
\begin{equation}\label{2.47}
	\begin{split}
		\frac{\xi h}{2}\frac{d}{dt} \int_{0}^{L}a(x)\int_{0}^{1}u^2(x,t-\rho h) \, d \rho dx =
		{}&-\int_{0}^{L} a(x)\frac{\xi}{2} \int_{0}^{1} \frac{d}{d\rho } (z(x,\rho,t))^2\, d\rho dx\\
		={}&-\int_{0}^{L} a(x)\frac{\xi}{2} [(z(x,1,t))^2-(z(x,0,t))^2  ] \,dx\\
		={}&\frac{\xi}{2}\int_{0}^{L} a(x)[(z(x,0,t))^2  -(z(x,1,t))^2] \,dx.
	\end{split}
\end{equation}
From \eqref{2.45}, \eqref{2.47}  and applying Young's inequality we obtain
\begin{align*}
	E_u'(t)\leq & -\frac{1}{2}u_{xx}^{2}(0)-\mu_1\int_{0}^{L}a(x) u^2(x,t)dx+\frac{\mu_2}{2}\int_{0}^{L}a(x)u^2(x,t)dx \nonumber\\
	&+\frac{\mu_2}{2}\int_{0}^{L}a(x)u^2(x,t-h)dx+\frac{\xi}{2h}\int_{0}^{L}a(x)u^2(x,t)dx-\frac{\xi}{2h}\int_{0}^{L}a(x)u^2(x,t-h)dx
\end{align*}
and the results holds directly from the previous estimate.
\end{proof}

\begin{proposition}\label{P1J}
	Assume that $a \in L^\infty(\Omega)$ and \eqref{2.13} holds. Then, the map
	\begin{equation}\label{J11}
	(u_0,z_0(\cdot,-h(\cdot)) \mapsto (u,z)
	\end{equation}
is continuous from $\mathcal{H}$ to $\mathcal{B}\times C([0,T],L^2((0,L)\times(0,1))$, and for $(u_0,z_0(\cdot,-h(\cdot))) \in \mathcal{H}$, the following estimates hold
\begin{equation}\label{J12}
\begin{split}
	\frac{1}{2}\int_{0}^{L}u^2(x,t) dx +\frac{\xi}{2}\int_{0}^{L}\int_{0}^{1}a(x)u^2(x,t-\rho h)d\rho dx \leq &\frac{1}{2}\int_{0}^{L}u_0^2(x)dx\\&+\frac{\xi}{2}\int_{0}^{L}\int_{0}^{1}a(x)z_0^2(x,-\rho h )d \rho dx
	\end{split}
\end{equation}
and
\begin{equation}\label{J13}
\begin{split}
	\|u_0\|_{L^2(0,L)}^2\leq & \frac{1}{T}\int_{0}^{T}\int_{0}^{L}u^2dxdt+ \int_{0}^{T}u_{xx}^{2}(0)dt  \\
	&+(2\mu_1+\mu_2)\int_{0}^{T}\int_{0}^{L}a(x)u^2dxdt+\int_{0}^{T}\int_{0}^{L}a(x)\mu_2u^2(x,t-h)dxdt.
	\end{split}
\end{equation}
\end{proposition}
\begin{proof}
First, note that \eqref{J12} follows from Proposition \ref{L24L}.
Now, let $p \in C^\infty([0,1]\times [0,T])$, $q \in C^\infty([0,L] \times [0,T])$ and $(u_0,z_0(\cdot,-h(\cdot)))$. Then multiplying $\eqref{J9}_4$ by $p(\rho,t)z(x,\rho,t)$, and using integration by parts we get
\begin{equation}\label{J16}
\begin{split}
\int_{0}^{1}\int_{0}^{L}&p(\rho,T)z^2(x,\rho,T)-p(\rho,0)z^2(x,\rho,0)dxd\rho-\frac{1}{h}\int_{0}^{T}\int_{0}^{1}\int_{0}^{L}(hp_t+p_\rho)z^2dxd\rho dt\\
&+\frac{1}{h}\int_{0}^{T}\int_{0}^{L}p(1,t)z^2(x,1,t)-p(0,t)z^2(x,0,t)dxdt=0.
\end{split}
\end{equation}
Taking $p(\rho,t)=1$ in \eqref{J16} we obtain
\begin{equation*}
	h\int_{0}^{1}\int_{0}^{L}z^2(x,\rho,T)-z^2(x,\rho,0)dxd\rho+\int_{0}^{T}\int_{0}^{L}z^2(x,1,t)-z^2(x,0,t)dxdt=0.
\end{equation*}
which implies that
\begin{align}\label{I5}
\frac{1}{2}\int_{0}^{T}\int_{0}^{L}u^2(x,t-h)dxdt\leq & \frac{1}{2}\int_{0}^{T}\int_{0}^{L}u^2(x,t)dxdt+\frac{h}{2} \int_{0}^{1} \int_{0}^{L}z_{0}^2(x,-\rho h)dxd\rho.
\end{align}

Now, multiplying $\eqref{J9}_1$ by $q(x,t)u(x,t)$ and integrating by parts we have 
\begin{equation}\label{J17}
\begin{split}
&	\frac{1}{2}\int_{0}^{L} q(x,T)u^2(x,T)dx-\frac{1}{2}\int_{0}^{L}q(x,0)u^2(x,0)dx\\
&	-\frac{1}{2}\int_{0}^{T}\int_{0}^{L} (q_t+q_x+q_{xxx}-q_{xxxxx})u^2dxdt+\frac{3}{2}\int_{0}^{T}\int_{0}^{L}q_xu_x^2dxdt\\
&	-\frac{5}{2}\int_{0}^{T}\int_{0}^{L}q_{xxx}u_x^2dxdt+\frac{5}{2}\int_{0}^{T}\int_{0}^{L}q_xu_{xx}^{2}dxdt+\frac{1}{2}\int_{0}^{T}q(0,t)u_{xx}^{2}(0,t)dt\\
&	+\int_{0}^{T}\int_{0}^{L}a(x)\mu_1qu^2dxdt+\int_{0}^{T}\int_{0}^{L}a(x)\mu_2qu(x,t-h)udxdt=0.
\end{split}
\end{equation}
Taking $q(x,t)=x$ in \eqref{J17} 
follows from \eqref{I5} that
\begin{equation*}
\begin{split}
	\frac{3}{2}\int_{0}^{T}\int_{0}^{L} u_x^2dx&+\frac{5}{2}\int_{0}^{T}\int_{0}^{L} u_{xx}^{2}dx = \frac{1}{2}\int_{0}^{L}x(u_{0}^{2}(x)-u^2(x,T))dx+\frac{1}{2}\int_{0}^{T}\int_{0}^{L}u^2 dxdt\\
	&-\int_{0}^{T}\int_{0}^{L}xa(x)\mu_1u^2 dxdt -\int_{0}^{T} \int_{0}^{L}xa(x)\mu_2u(x,t-h)u(x,t)dxdt\\
	\leq & \frac{L}{2}\|u_0\|_{L^2(0,L)}^2+\frac{L}{2}h\|a\|_\infty\mu_2 \int_{0}^{1} \int_{0}^{L}z_{0}^2(x,-\rho h)dxd\rho\\
	&+\left(\frac{1}{2}+L\|a\|_\infty(\mu_1+\mu_2)\right)T \left( \int_{0}^{L}u_0^2(x)dx+\xi\|a\|_{\infty}\int_{0}^{L}\int_{0}^{1}z_0^2(x,-\rho h )d \rho dx \right)\\
 \leq &C(a,h,\mu_1,\mu_2,\xi,L)(1+T)\|u_0,z_{0}(\cdot,-h(\cdot)\|_{L^2(0,L) \times L^2(0,1)}^2,
 \end{split}
\end{equation*}
where
\begin{equation*}
	C(a,h,\mu_1,\mu_2,\xi,L)=\left(\frac{L}{2}+\frac{Lh \mu_2}{2\xi}+1+2L\|a\|_\infty(\mu_1+\mu_2)\right).
\end{equation*}

Finally,  choosing $q(x,t)=T-t$ in \eqref{J17} we obtain
\begin{equation*}
	\begin{gathered}
-\frac{1}{2}\int_{0}^{L}Tu_0^2(x)dx+\frac{1}{2}\int_{0}^{T}\int_{0}^{L}u^2dxdt+\frac{1}{2}\int_{0}^{T}(T-t)u_{xx}^{2}(0)dt\\
+\int_{0}^{T}\int_{0}^{L}(T-t)a(x)\mu_1u^2dxdt+\int_{0}^{T}\int_{0}^{L}(T-t)a(x)\mu_2u(x,t)u(x,t-h)dxdt=0.
	\end{gathered}
\end{equation*}
Therefore,
\begin{align*}
\|u_0\|_{L^2(0,L)}^2
\leq & \frac{1}{T}\int_{0}^{T}\int_{0}^{L}u^2dxdt+ \int_{0}^{T}u_{xx}^{2}(0)dt\\
&+(2\mu_1+\mu_2)\int_{0}^{T}\int_{0}^{L}a(x)u^2dxdt+\|a\|_\infty \mu_2 \int_{0}^{T}\int_{0}^{L}u^2(x,t)dxdt\\
&+\frac{h}{2}\|a\|_\infty \mu_2 \int_{0}^{1} \int_{0}^{L}z_{0}^2(x,-\rho h)dxd\rho,
\end{align*}
showing \eqref{J13}, and the proof is complete.
\end{proof}

\subsection{Linear system with source term} Consider the higher order KdV linear equation with a source term $f(x,t)$, in the right hand side:
\begin{equation}
	\begin{cases}\label{J19}
u_{t}(x,t)+u_{x}(x,t)+u_{xxx}(x,t)-u_{xxxxx}(x,t)\\ \quad  \quad  \quad  \quad  \quad  \quad  \quad  +a\left(x\right)(\mu_1 u(x,t)+\mu_2u(x,t-h))=f(x,t) & x \in (0,L),~ t>0,\\
		u\left(  0,t\right)  =u\left(  L,t\right)  =u_{x}\left(  0,t\right)
		=u_{x}\left(  L,t\right)  =u_{xx}\left(  L,t\right)  =0 & t>0,\\
		u\left(  x,0\right)  =u_{0}\left(  x\right)  & x \in (0,L),\\
		u(x,t)=z_0(x,t) & x \in (0,L),~t \in (-h,0),
	\end{cases}
\end{equation}
where $\mu_1>\mu_2$ and $a(x)$ satisfies the same hypothesis of the previous section. The next result deal with the existence of solution of this system.

\begin{proposition}\label{P2J}
	Assume that $a(x)\in L^\infty(\Omega)$ and \eqref{2.7} holds. For any $(u_0,z_0(\cdot,-h(\cdot)) \in \mathcal{H}$ and $f \in L^2(0,T,L^2(0,L))$, there exists a unique mild solution for \eqref{J19} in the class $$(u,u(\cdot,t-h(\cdot))) \in \mathcal{B} \times C([0,T],L^2((0,L)\times(0,1))).$$ Moreover, we have the following estimates
	\begin{equation}\label{J20}
		\|(u,z)\|_{C([0,T],\mathcal{H})} \leq e^{\frac{\xi \|a\|_\infty}{2h} T} \left(\|(u_0,z_0(\cdot,-h(\cdot)))\|_{\mathcal{H}}+\|f\|_{L^1(0,T,L^2(0,L))} \right)
	\end{equation}
and
\begin{equation}\label{J21}
\|u\|_{L^2(0,T,H^2(0,L))}^{2}\leq C(1+T+e^{\frac{\xi\|a\|_{\infty}}{h}T})\left(\|(u_0,z_0(\cdot,-h(\cdot)))\|_{\mathcal{H}}^{2}+\|f\|_{L^1(0,T,L^2(0,L))}^{2} \right),
\end{equation}
where
\begin{equation*}
	C=C(a,h,\mu_1,\mu_2,\xi,L)=\left(\frac{3L}{2}+\frac{Lh \mu_2}{2\xi}+1+2L\|a\|_\infty(\mu_1+\mu_2)\right).
\end{equation*}
\end{proposition}
\begin{proof}Thanks to the fact that $\mathcal{A}$ is an infinitesimal generator of a $C_0$-semigroup $(e^{t\mathcal{A}})$ satisfying $\|e^{t\mathcal{A}}\|_{\mathcal{L}(\mathcal{H})} \leq e^{\frac{\xi \|a\|_\infty}{2h}t}$ and together with the fact that we can rewrite system \eqref{J19} as a first order system (see \eqref{J10}) with source term $(f(\cdot,t),0)$, we have that  \eqref{J19} is well-posed in $C([0,T],\mathcal{H})$. Additionally,  the proof of \eqref{J21} follows exactly the same steps of the proof of Proposition \ref{P1J}. However, we have to be careful to the fact that the right hand side terms are not homogeneous anymore, so, we need to note that
\begin{align*}
	\left| \int_{0}^{T}\int_{0}^{L} x f(x,t) u(x,t) dxdt\right| 
	\leq & \frac{L}{2}\|u\|_{C([0,T],L^2(0,L))}^2+\frac{L}{2}\|f\|_{L^1(0,T,L^2(0,L))}^2,
\end{align*}
and the result is achieved. 
\end{proof}

\subsection{Nonlinear system: Global results\label{ss2.3}} In this section we prove the global well-posedness result for the nonlinear system \eqref{J1}. 
The first step is to show that the nonlinear term $uu_x$ can be considered as a source term of the linear equation \eqref{J19}. Precisely, the result is the following.
\begin{proposition}\label{P3J} 
Let $u \in \mathcal{B}$. Then $uu_x \in L^1(0,T,L^2(0,L))$ and the map
\begin{equation*}
	u \in \mathcal{B} \mapsto uu_x \in L^1(0,T,L^2(0,L))
\end{equation*}
is continuous. In particular, there exists $K>0$ ($K=\sqrt{2}$) such that, for any $u,v \in \mathcal{B}$, we have
\begin{equation*}
	\int_{0}^{T} \|uu_x-vv_x\|_{L^2(0,L)}dt \leq KT^{\frac{1}{4}}(\|u\|_{\mathcal{B}}+\|v\|_{\mathcal{B}})\|u-v\|_{\mathcal{B}}
\end{equation*}
\end{proposition}
\begin{proof} The proof is a variant of \cite[ Proposition 4.1]{Rosier1} so, we just give a sketch of the proof.  First, note that for $z \in H_{0}^{2}(0,L)$ we have
\begin{equation}\label{G4}
\|z\|_{L^\infty(0,L)}^2 \leq 2 \|z\|_{L^2(0,L)}\|z'\|_{L^2(0,L)}.
\end{equation}
From Hölder's inequality and \eqref{G4} we obtain
\begin{equation}\label{G5}
\begin{split}
	\|z\|_{L^2(0,T,L^\infty(0,L))} 
	\leq & \sqrt{2}T^{\frac{1}{4}}\|z\|_{L^\infty(0,T,L^2(0,L))}^{\frac{1}{2}}\|z\|_{L^\infty(0,T,H_{0}^2(0,L))}^{\frac{1}{2}}
	\end{split}
\end{equation}
Let $u,z \in \mathcal{B}$, then from \eqref{G5} it follows that
\begin{align*}
	\|uu_x -vv_x\|_{L^1(0,T,L^2(0,L))} 
\leq & \frac{\sqrt{2}}{2}T^{\frac{1}{4}} \|u\|_{L^2(0,T,H_{0}^{2}(0,L))} \left(\|u-v\|_{L^\infty(0,T,L^2(0,L))}+\|u-v\|_{L^2(0,T,H_{0}^{2}(0,L))}\right)\\
&+\frac{\sqrt{2}}{2}T^{\frac{1}{4}} \|u-v\|_{L^2(0,T,H_{0}^{2}(0,L))} \left(\|v\|_{L^\infty(0,T,L^2(0,L))}+\|v\|_{L^2(0,T,H_{0}^{2}(0,L))}\right)\\
\leq & \sqrt{2} T^{\frac{1}{4}} \left( \|u\|_{\mathcal{B}}+\|v\|_{\mathcal{B}}\right)\|u-v\|_{\mathcal{B}},
\end{align*}
and the proof is complete.
\end{proof}

We are now in position to prove the global existence of solutions of \eqref{J1}.
\begin{proposition}\label{P4J}
Let $L>0$ and assume that $a(x) \in L^\infty(\Omega)$ and \eqref{2.7} holds. Then, for every $(u_0,z_0(\cdot,-h(\cdot))) \in \mathcal{H}$, there exists a unique $u \in \mathcal{B}$ solution of system \eqref{J1}. Moreover, there exists $C>0$ such that
\begin{equation}\label{J22}
\|u_x\|_{L^2(0,T,L^2(0,L))}^{2}+\|u_{xx}\|_{L^2(0,T,L^2(0,L))}^2 \leq C(\|(u_0,z_0(\cdot,-h(\cdot)))\|_{\mathcal{H}}^{2}+\|(u_0,z_0(\cdot,-h(\cdot)))\|_{\mathcal{H}}^{4}).
\end{equation}
\end{proposition}
\begin{proof} To prove this result we can follow a standard argument in the literature (see e.g.  \cite{Zuazua1,Pazoto}).  Thus, our goal is to obtain the global existence of solutions proving the local existence and using the following \textit{a priori} estimate
\begin{equation}\label{J23}
\|(u(\cdot,t),u(\cdot,t-h(\cdot)))\|_{\mathcal{H}}^{2} \leq e^{\frac{\xi\|a\|_\infty}{2}t} \|(u_0,z_0(\cdot,-h(\cdot)))\|_{\mathcal{H}}^{2}.
\end{equation}
Then, we are concentrated to prove the local existence and uniqueness of solutions to \eqref{J1}. Let $(u_0,z_0(\cdot,-h(\cdot))) \in \mathcal{H}$ and $u \in \mathcal{B}$, we consider the map $\Phi: \mathcal{B} \rightarrow \mathcal{B}$ defined by $\Phi(u)=\tilde{u}$ where $\tilde{u}$ is solution of 
\begin{equation*}
	\begin{cases}
		\tilde{u}_{t}(x,t)+\tilde{u}_{x}(x,t)+\tilde{u}_{xxx}(x,t)-\tilde{u}_{xxxxx}(x,t)\\ \quad  \quad  \quad  \quad  \quad  \quad  \quad  +a\left(x\right)(\mu_1 \tilde{u}(x,t)+\mu_2\tilde{u}(x,t-h))=-u(x,t)u_x(x,t) & x \in (0,L),~ t>0,\\
		\tilde{u}\left(  0,t\right)  =\tilde{u}\left(  L,t\right)  =\tilde{u}_{x}\left(  0,t\right)
		=\tilde{u}_{x}\left(  L,t\right)  =\tilde{u}_{xx}\left(  L,t\right)  =0 & t>0,\\
		\tilde{u}\left(  x,0\right)  =u_{0}\left(  x\right)  & x \in (0,L),\\
		\tilde{u}(x,t)=z_0(x,t) & x \in (0,L),~t \in (-h,0).
	\end{cases}
\end{equation*}
Notice that $u \in \mathcal{B}$ is solution of \eqref{J1} if and only if $u$ is a fixed point of the map $\Phi$.  Thus, let us prove that the map $\Phi$ is a contraction.  

In fact, thanks to \eqref{J20}, \eqref{J21} and Proposition \ref{P3J}, we get
\begin{align*}
	\|\Phi u\|_{\mathcal{B}} 
	\leq & \sqrt{C}(1+\sqrt{T}+e^{\frac{\xi\|a\|_{\infty}}{2h}T})\left(\|(u_0,z_0(\cdot,-h(\cdot)))\|_{\mathcal{H}}+\|uu_x\|_{L^1(0,T,L^2(0,L))} \right)\\
	\leq &  \sqrt{C}(1+\sqrt{T}+e^{\frac{\xi\|a\|_{\infty}}{2h}T})\left(\|(u_0,z_0(\cdot,-h(\cdot)))\|_{\mathcal{H}}+KT^{\frac{1}{4}}\|u\|_{\mathcal{B}}^2\right)\\
	\leq & \sqrt{C}(1+\sqrt{T}+e^{\frac{\xi\|a\|_{\infty}}{2h}T})\|(u_0,z_0(\cdot,-h(\cdot)))\|_{\mathcal{H}}+ \sqrt{C}K(2T^{\frac{1}{4}}+T^{\frac{1}{4}}e^{\frac{\xi\|a\|_{\infty}}{2h}T})\|u\|_{\mathcal{B}}^{2},
\end{align*}
if $T<1$. Moreover, for the same reasons, we have
\begin{align*}
	\|\Phi(u_1)-\Phi(u_2)\|_{\mathcal{B}} 
	\leq&  \sqrt{C}K(1+\sqrt{T}+e^{\frac{\xi\|a\|_{\infty}}{2h}T})T^{\frac{1}{4}}\left(\|u_1\|_{\mathcal{B}}+\|u_2\|_{\mathcal{B}} \right)\|u_1-u_2\|_{\mathcal{B}}.
\end{align*}
Now, consider $\Phi$ restricted to the closed ball $\{u \in \mathcal{B}: \|u\|_{\mathcal{B}}\leq R\}$ with $R>0$ to be chosen later. Then,
\begin{equation*}
\|\Phi(u)\|_{\mathcal{B}} \leq 	\sqrt{C}(1+\sqrt{T}+e^{\frac{\xi\|a\|_{\infty}}{2h}T})\|(u_0,z_0(\cdot,-h(\cdot)))\|_{\mathcal{H}}+ \sqrt{C}K(2T^{\frac{1}{4}}+T^{\frac{1}{4}}e^{\frac{\xi\|a\|_{\infty}}{2h}T})R^{2}
\end{equation*}
and
\begin{equation*}
	\|\Phi(u_1)-\Phi(u_2)\|_{\mathcal{B}} \leq 2\sqrt{C}K(1+\sqrt{T}+e^{\frac{\xi\|a\|_{\infty}}{2h}T})T^{\frac{1}{4}}R\|u_1-u_2\|_{\mathcal{B}}.
\end{equation*}
So, pick $R=4\sqrt{C}\|(u_0,z_0(\cdot,-h(\cdot)))\|_{\mathcal{H}}$ and $T>0$ satisfying
\begin{equation*}
	\begin{cases}
		\sqrt{T}+8\sqrt{C}KT^{\frac{1}{4}}+4\sqrt{C}KT^{\frac{1}{4}}e^{\frac{\xi\|a\|_\infty}{2h}T}<1,\\ 2T^{\frac{1}{4}}+T^{\frac{1}{4}}e^{\frac{\xi\|a\|_\infty}{2h}T}<\frac{1}{2\sqrt{C}KR}, \\
		T<1,~e^{\frac{\xi\|a\|_\infty}{2h}T}<2,
	\end{cases}
\end{equation*}
then $\|\Phi(u)\|_{\mathcal{B}} <R$ and $\|\Phi(u_1)-\Phi(u_2)\|_{\mathcal{B}} \leq C_1\|u_1-u_2\|_{\mathcal{B}}$, with $C_1<1$, showing that $\Phi$ is a contraction. Consequently, we can apply the Banach fixed point theorem and the map $\Phi$ has a unique fixed point.

In this last part, let us show \eqref{J22}.  
Following the same steps of the proof of Proposition \ref{P1J}, that is,  multiplying \eqref{J1} $xu$, integrating by parts and using \eqref{J23}, we obtain
\begin{equation*}
	\frac{3}{2}\int_{0}^{T}\int_{0}^{L} u_x^2dx+\frac{5}{2}\int_{0}^{T}\int_{0}^{L} u_{xx}^{2}dx  \leq C(1+T)\|(u_0,z_0(\cdot,-h(\cdot)))\|_{\mathcal{H}}^{2}+\frac{1}{3}\int_{0}^{T}\int_{0}^{L} u^3(x,t)dxdt.
\end{equation*}
As $H^1(0,L) \hookrightarrow C([0,L])$ we obtain, by using the Cauchy-Schwarz inequality and \eqref{J23}, that
\begin{equation}\label{uuu}
\begin{split}
	\int_{0}^{T} \int_{0}^{L}|u(x,t)|^3dxdt \leq& \int_{0}^{T} \|u\|_{L^\infty(0,L)}\int_{0}^{L}u^2(x,t)dxdt\\
	\leq & \sqrt{L} \int_{0}^{T}\|u(\cdot,t)\|_{H^1(0,L)}\int_{0}^{L}u^2(x,t)dxdt\\
	\leq & \sqrt{LT}\|u\|_{L^\infty(0,T,L^2(0,L))}^{2}\|u\|_{L^2(0,T,H^1(0,L))}\\
	\leq & \sqrt{LT} \|(u_0,z_0(\cdot,-h(\cdot)))\|_{\mathcal{H}}^{2}\|u\|_{L^2(0,T,H^1(0,L))}.
	\end{split}
\end{equation}
Consequently, we obtain
\begin{align*}
	\frac{3}{2}\int_{0}^{T}\int_{0}^{L} u_x^2dx+\frac{5}{2}\int_{0}^{T}\int_{0}^{L} u_{xx}^{2}dx  \leq & C(1+T)\|(u_0,z_0(\cdot,-h(\cdot)))\|_{\mathcal{H}}^{2}+\frac{\sqrt{LT}}{4\varepsilon}\|(u_0,z_0(\cdot,-h(\cdot)))\|_{\mathcal{H}}^{4}\\
	&+\varepsilon \sqrt{LT}\|u\|_{L^2(0,T,H^1(0,L))}^{2}.
\end{align*}
For $\varepsilon>0$ small enough we obtain
\begin{equation*}
	\frac{1}{2}\int_{0}^{T}\int_{0}^{L} u_x^2dx+\frac{5}{2}\int_{0}^{T}\int_{0}^{L} u_{xx}^{2}dx  \leq C(1+T)\|(u_0,z_0(\cdot,-h(\cdot)))\|_{\mathcal{H}}^{2}+\frac{\sqrt{LT}}{4\varepsilon}\|(u_0,z_0(\cdot,-h(\cdot)))\|_{\mathcal{H}}^{4},
\end{equation*}
which completes the proof.
\end{proof}

\section{Study of the damping--delayed system}\label{Sec3}
In this section we are interested in studying the time--delayed system \eqref{fd1}.  Note that if we choose  $a(x):= \mu_1 a(x)$ and $b(x):= \mu_2 a(x)$, where $\mu_1$ and $\mu_2$ are real constants, in the system \eqref{fd1} we recovered the system \eqref{J1}.  Thus, in this section, we deal with the exponential stability of the solutions associated to the system  \eqref{J1} considering the case when $\supp b \not\subset \supp a$. In this case, the derivative of the energy $E$ defined by
\begin{equation}\label{J61}
	E_u(t)=\frac{1}{2}\int_0^Lu^2(x,t)dx+\frac{h}{2}\int_0^L\int_0^1 b(x)u^2(x,t-\rho h)\rho dx,
\end{equation}
satisfies
\begin{align*}
	\frac{d}{dt}E_u(t) =& -u_{xx}^{2}(0)-\int_{0}^{L}a(x)u^2(x,t)dx-\int_{0}^{L}b(x)u(x,t)u(x,t-h)dx\\
	&+\frac{1}{2}\int_{0}^{L}b(x)u^2(x,t)dx-\frac{1}{2}\int_{0}^{L}b(x)u^2(x,t-h)dx\\
	\leq& -u_{xx}^{2}(0)-\int_{0}^{L}a(x)u^2(x,t)dx+\frac{1}{2}\int_{0}^{L}b(x)u^2(x,t)dx+\frac{1}{2}\int_{0}^{L}b(x)u^2(x,t-h)dx\\
	&+\frac{1}{2}\int_{0}^{L}b(x)u^2(x,t)dx-\frac{1}{2}\int_{0}^{L}b(x)u^2(x,t-h)dx\\
	\leq & \int_{0}^{L}b(x)u^2(x,t)dx. 
\end{align*}
The previous inequality means that the energy is not decreasing in general, since the term $b(x)\geq 0$ on $(0,L)$. So, inspired by \cite{Valein}, we consider the following perturbation system
\begin{equation}
	\begin{cases}\label{J42}
		u_{t}(x,t)+u_{x}(x,t)+u_{xxx}(x,t)-u_{xxxxx}(x,t)+u(x,t)u_{x}(x,t)\\ \quad  \quad  \quad  \quad  \quad  \quad  +b(x)u(x,t-h)+a\left(x\right)u(x,t)+\xi b(x)u(x,t)=0 & x \in (0,L),~ t>0,\\
		u\left(  0,t\right)  =u\left(  L,t\right)  =u_{x}\left(  0,t\right)
		=u_{x}\left(  L,t\right)  =u_{xx}\left(  L,t\right)  =0 & t>0,\\
		u\left(  x,0\right)  =u_{0}\left(  x\right)  & x \in (0,L),\\
		u(x,t)=z_0(x,t) & x \in (0,L),~t \in (-h,0),
	\end{cases}
\end{equation}
which is ``close'' to \eqref{fd1} but with a decreasing energy, with $\xi$ a positive constant. So, considering the energy defined by
\begin{equation}\label{J43}
	E_u(t)=\frac{1}{2}\int_0^Lu^2(x,t)dx+\frac{\xi h}{2}\int_0^L\int_0^1 b(x)u^2(x,t-\rho h)\rho dx,
\end{equation}
we get, for $\xi>1$, that the derivative of the energy $E_u(t)$, for classical solutions of \eqref{J42}, satisfies
\begin{align*}
	\frac{d}{dt}E_u(t)
	\leq & -u_{xx}^{2}(0)-\int_{0}^{L}a(x)u^2(x,t)dx+\frac{1}{2}\int_{0}^{L}b(x)u^2(x,t)dx+\frac{1}{2}\int_{0}^{L}b(x)u^2(x,t-h)dx\\
	&-\int_{0}^{L}\xi b(x)u^2(x,t)dx+\frac{1}{2}\int_{0}^{L}\xi b(x)u^2(x,t)dx-\frac{1}{2}\int_{0}^{L}\xi b(x)u^2(x,t-h)dx\\
	\leq & -u_{xx}^{2}(0)-\int_{0}^{L}a(x)u^2(x,t)dx+\frac{1}{2}\int_{0}^{L}(b(x)-\xi b(x))u^2(x,t)dx\\
	&+\frac{1}{2}\int_{0}^{L}(b(x)-\xi b(x))u^2(x,t-h)dx \leq 0.
\end{align*}

\subsection{Local stability: A perturbation argument}
In this subsection, before to present the main result of this section, we will study the asymptotic stability of the linear system associated to \eqref{J42}, namely, 
\begin{equation}
	\begin{cases}\label{J8}
		u_{t}(x,t)+u_{x}(x,t)+u_{xxx}(x,t)-u_{xxxxx}(x,t)+a\left(x\right)u(x,t)\\ \quad  \quad  \quad  \quad  \quad  \quad  \quad  \quad  \quad  \quad \quad +b(x)u(x,t-h)=0 & x \in (0,L),~ t>0,\\
		u\left(  0,t\right)  =u\left(  L,t\right)  =u_{x}\left(  0,t\right)
		=u_{x}\left(  L,t\right)  =u_{xx}\left(  L,t\right)  =0 & t>0,\\
		u\left(  x,0\right)  =u_{0}\left(  x\right)  & x \in (0,L),\\
		u(x,t)=z_0(x,t) & x \in (0,L),~t \in (-h,0),
	\end{cases}
\end{equation}
 in the case $\supp b \not\subset \supp a$ and $\xi>1$.  Note that this system can write as the first order system
 \begin{equation}\label{J10a}
	\begin{cases}
		\displaystyle \frac{\partial U}{\partial t}(t)=\mathcal{A}U(t),\\
		U(0)={}  (u_{0}(x), z_0(x,-\rho h )).
	\end{cases}
\end{equation}
 where the corresponding operator $\mathcal{A}$ is defined by
\begin{equation*}
	\mathcal{A}=	\mathcal{A}_0+	B
\end{equation*}
with domain $\mathscr{D}(\mathcal{A})=\mathscr{D}(\mathcal{A}_0)$ and the bounded operator $B$ is defined by
\begin{equation*}
	B(u,z)=(\xi b(x)u,0) \hbox{ for all } (u,z) \in \mathcal{H}.
\end{equation*}
Here,  $\mathcal{A}_0$ is defined by \eqref{A0}.  The first result ensures that the system \eqref{J8} is well-posed.  It is consequence of the analysis made for an auxiliary system in Appendix \ref{ApA}.
\begin{proposition}\label{P8Ja}
Assume that $a(x)$ and $b(x)$ are nonnegative function in $L^\infty(0,L)$, $b(x)\geq b_0>0$ in $\omega$, $L<\pi \sqrt{3}$ and $\xi>1$. Then, for every $(u_0,z_0(\cdot,-h(\cdot))) \in \mathcal{H}$, there exists a unique mild solution $U \in C([0,\infty),\mathcal{H})$ for system \eqref{J8}. Additionally, for every $U_0 \in \mathscr{D}(\mathcal{A})$, the solution is classical and satisfies $$U \in C([0,\infty),\mathscr{D}(\mathcal{A})) \cap C^1([0,\infty),\mathcal{H}).$$
\end{proposition}
\begin{proof}
Assume that $\|b\|_{\infty} \leq 1$. From Theorem \ref{T4.2J} we have that 
\begin{align*}
\left((\mathcal{A}_0+B)U,U\right)_{\mathcal{H}}
\leq& \frac{(3\xi+1)}{2}\|U\|_{\mathcal{H}}^{2}.
\end{align*}
for all $U \in \mathscr{D}(\mathcal{A})$.  In the same way, we obtain 
\begin{equation*}
\left(\left(\mathcal{A}_0+B\right)^*U,U\right)_{\mathcal{H}} \leq \frac{(3\xi+1)}{2}\|U\|_{\mathcal{H}}^{2}
\end{equation*}
for all $U \in \mathscr{D}(\mathcal{A}^{*})$.

 Finally, since for $\lambda=\frac{(3\xi+1)}{2}$, $\mathcal{A}-\lambda I$ is a densely defined closed linear operator, and both $\mathcal{A}-\lambda I$ and $(\mathcal{A}-\lambda I)^*$ are dissipative, then $\mathcal{A}$ is the infinitesimal generator of a $C_0$-semigroup on $\mathcal{H}$ satisfying $\|e^{t\mathcal{A}}\|_{\mathcal{L}(\mathcal{H})} \leq e^{\frac{(3\xi+1)}{2}t}$.
\end{proof}

The next result ensures that the energy
	\begin{equation}\label{J43.1}
		E_u(t)=\frac{1}{2}\int_0^Lu^2(x,t)dx+\frac{h}{2}\int_0^L\int_0^1 b(x)u^2(x,t-\rho h)\rho dx,
	\end{equation}
	associated of the system \eqref{J8}
decays exponentially, and it is a consequence of the analysis made in the Appendix \ref{ApA}.
	\begin{proposition}\label{P8J}
		Assume that $a$ and $b$ are nonnegative function in $L^\infty(0,L)$, $b(x)\geq b_0>0$ in $\omega$, $L<\pi \sqrt{3}$ and $\xi>1$. So, there exists $\delta>0$ (depending on $\xi,L,h$) such that is, $\|b\|_{\infty} \leq \delta$ then, for every $(u_0,z_0(\cdot,-h(\cdot))) \in \mathcal{H}$ the energy of system $E_u$, defined in \eqref{J43.1}, is exponentially stable.  More precisely, there exists $T_0>0$ and two positive constants $\nu$ and $C$ such that 
		\begin{equation*}
			E_u(t) \leq C e^{-\nu t} E_u(0), \hbox{ for all } t > T_0.
		\end{equation*}
\end{proposition}
\begin{proof}
To prove this result, let us consider the two systems
\begin{equation}
	\begin{cases}\label{V1}
		v_{t}(x,t)+v_{x}(x,t)+v_{xxx}(x,t)-v_{xxxxx}(x,t)+a\left(x\right)v(x,t)\\ \quad  \quad  \quad  \quad  \quad  \quad  \quad  \quad  \quad  \quad \quad+b(x)z^1(1)+\xi b(x)v(x,t)=0 & x \in (0,L),~ t>0,\\
		v\left(  0,t\right)  =v\left(  L,t\right)  =v_{x}\left(  0,t\right)
		=v_{x}\left(  L,t\right)  =v_{xx}\left(  L,t\right)  =0 & t>0,\\
		v\left(  x,0\right)  =u_{0}\left(  x\right)  & x \in (0,L),\\
		h z^1_t(x,\rho,t)+z^1_\rho(x,\rho,t)=0 & x \in (0,L), ~\rho \in (0,1), ~ t>0,\\
		z^1(x,0,t)=v(x,t) & x \in (0,L),~ t>0,\\
		z^1(x,\rho,0)=v(x,-\rho h)=z_0(x,-\rho h) & x \in (0,L), ~\rho \in (0,1)
	\end{cases}
\end{equation}
and
\begin{equation}
	\begin{cases}\label{V2}
		w_{t}(x,t)+w_{x}(x,t)+w_{xxx}(x,t)-w_{xxxxx}(x,t)\\ \quad  \quad  \quad  \quad  \quad  \quad  \quad  \quad  +a\left(x\right)w(x,t)+b(x)z^2(1)=\xi b(x)v(x,t) & x \in (0,L),~ t>0,\\
		w\left(  0,t\right)  =w\left(  L,t\right)  =w_{x}\left(  0,t\right)
		=w_{x}\left(  L,t\right)  =w_{xx}\left(  L,t\right)  =0 & t>0,\\
		w\left(  x,0\right)  =0  & x \in (0,L),\\
		h z^2_t(x,\rho,t)+z^2_\rho(x,\rho,t)=0 & x \in (0,L), ~\rho \in (0,1), ~ t>0,\\
		z^2(x,0,t)=w(x,t) & x \in (0,L),~ t>0,\\
		z^2(x,\rho,0)=0 & x \in (0,L), ~\rho \in (0,1).
	\end{cases}
\end{equation}
Define $u=v+w$ and $z=z^1+z^2$, then
\begin{equation}
	\begin{cases}\label{V3}
		u_{t}(x,t)+u_{x}(x,t)+u_{xxx}(x,t)-u_{xxxxx}(x,t)\\ \quad  \quad  \quad  \quad  \quad  \quad  \quad  \quad  \quad  \quad \quad+a\left(x\right)u(x,t)+b(x)z(1)=0 & x \in (0,L),~ t>0,\\
		u\left(  0,t\right)  =u\left(  L,t\right)  =u_{x}\left(  0,t\right)
		=u_{x}\left(  L,t\right)  =u_{xx}\left(  L,t\right)  =0 & t>0,\\
		u\left(  x,0\right)  =u_{0}\left(  x\right)  & x \in (0,L),\\
		h z_t(x,\rho,t)+z_\rho(x,\rho,t)=0 & x \in (0,L), ~\rho \in (0,1), ~ t>0,\\
		z(x,0,t)=u(x,t) & x \in (0,L),~ t>0,\\
		z(x,\rho,0)=z_0(x,-\rho h) & x \in (0,L), ~\rho \in (0,1).
	\end{cases}
\end{equation}
Fix $0<\eta<1$ and pick 
$$
T_0=\frac{1}{2\gamma}\ln\left( \frac{2\xi \kappa }{\eta}\right)+1,
$$
so  $\kappa e^{-2 \gamma T_0}<\frac{\eta}{2 \xi},$ where $\alpha, \beta, \gamma, \kappa$ are given according to Proposition \ref{P7J}. As we have that 
$E_v(0)\leq \xi E_u(0),$
we obtain
$$E_v(T_0)\leq \kappa e^{-2\gamma T_0} E_v(0) \leq \frac{\eta}{2\xi}E_v(0)\leq \frac{\eta}{2} E_u(0).$$ 

Now, consider $\varepsilon>0$ such that $0<\eta+\varepsilon<1$ and 
\begin{equation*}
	\|b\|_{\infty} \leq \min \left\{ \frac{\sqrt{\varepsilon}}{\sqrt{\xi^3}\kappa^{\frac{1}{2}}e^{\frac{(3\xi+1)}{2}\left(\frac{1}{2\gamma}\ln\left( \frac{2\xi \kappa }{\eta}\right)+2\right) }},1\right\}.
\end{equation*}
Therefore,
\begin{align*}
	E_u(T_0)
	\leq & \int_{0}^{L}v^2(x,T_0)dx+h\xi\int_{0}^{L}\int_{0}^{1}b(x)v^2(x,T_0-\rho h)d\rho dx\\
	&+\int_{0}^{L}w^2(x,T_0)dx+h\xi \|b\|_\infty \int_{0}^{L}\int_{0}^{1}w^2(x,T_0-\rho h)d\rho dx\\
	\leq & 2E_v(T_0)+\|(w(T_0,w(\cdot,T_0-h(\cdot))))\|_{\mathcal{H}}.
\end{align*}
Noting that 
\begin{equation*}
	(w(T_0),w(\cdot,T_0-h(\cdot))) =\int_{0}^{T_0}e^{\mathcal{A}(t-s)}(\xi b(x)v,0)ds,
\end{equation*}
we get
\begin{align*}
	 \|(w(T_0),w(\cdot,T_0-h(\cdot)))\|_{\mathcal{H}} 
	 	 \leq& \int_{0}^{T_0} e^{\frac{(3\xi+1)}{2}(T_0-s)}\left(\int_{0}^{L} |\xi b(x) v|^2 dx \right)^{\frac{1}{2}}ds\\
	 	 \leq & \sqrt{2}\xi \|b\|_{\infty} \int_{0}^{T_0} e^{\frac{(3\xi+1)}{2}(T_0-s)} \kappa^{\frac{1}{2}}e^{-\gamma s}E_{v}^{\frac{1}{2}}(0)ds\\
	 	 \leq & \sqrt{2} \xi \|b\|_{\infty} \kappa^{\frac{1}{2}} E_{v}^{\frac{1}{2}}(0) \int_{0}^{T_0}e^{\frac{(3\xi+1)}{2}(T_0-s)} e^{-\gamma s}ds\\
	 	 \leq&  2 \xi^2 \|b\|_{\infty}^{2} e^{(3\xi+1)T_0} \kappa E_{v}(0),
\end{align*}
where we have used that 
\begin{align*}
 \int_{0}^{T_0}e^{\frac{(3\xi+1)}{2}(T_0-s)} e^{-\gamma s}ds 
= 	& \frac{e^{\frac{(3\xi+1)}{2}T_0}-e^{-\gamma T_0} }{\frac{(3\xi+1)}{2}+\gamma} \quad \text{and} \quad \frac{(3\xi+1)}{2}+\gamma>2 .
\end{align*}
Therefore, by the previous inequality we have
\begin{equation*}
	E_u(T_0)\leq \eta E_u(0)+2 \xi^3 \|b\|_{\infty}^{2} e^{(3\xi+1)T_0} \kappa E_{u}(0)<(\eta+\varepsilon)E_u(0).
\end{equation*}

Finally, for $T_0>0$ defined in \eqref{To}, let us consider the following two systems
\begin{equation}
	\begin{cases}\label{V4}
		v_{t}(x,t)+v_{x}(x,t)+v_{xxx}(x,t)-v_{xxxxx}(x,t)+a\left(x\right)v(x,t)\\ \quad  \quad  \quad  \quad  \quad  \quad  \quad  \quad  \quad  \quad \quad+b(x)z^1(1)+\xi b(x)v(x,t)=0 & x \in (0,L),~ t>0,\\
		v\left(  0,t\right)  =v\left(  L,t\right)  =v_{x}\left(  0,t\right)
		=v_{x}\left(  L,t\right)  =v_{xx}\left(  L,t\right)  =0 & t>0,\\
		v\left(  x,0\right)  =u\left(  x, T_0\right)  & x \in (0,L),\\
		h z^1_t(x,\rho,t)+z^1_\rho(x,\rho,t)=0 & x \in (0,L), ~\rho \in (0,1), ~ t>0,\\
		z^1(x,0,t)=v(x,t) & x \in (0,L),~ t>0,\\
		z^1(x,\rho,0)=z(x,\rho, T_0) & x \in (0,L), ~\rho \in (0,1)
	\end{cases}
\end{equation}
and
\begin{equation}
	\begin{cases}\label{V5}
		y_{t}(x,t)+y_{x}(x,t)+y_{xxx}(x,t)-y_{xxxxx}(x,t)+a\left(x\right)y(x,t)\\ \quad  \quad  \quad  \quad  \quad  \quad  \quad  \quad  \quad  \quad \quad+b(x)z^2(1)=\xi b(x)v(x,t) & x \in (0,L),~ t>0,\\
		y\left(  0,t\right)  =y\left(  L,t\right)  =y_{x}\left(  0,t\right)
		=y_{x}\left(  L,t\right)  =y_{xx}\left(  L,t\right)  =0 & t>0,\\
		y\left(  x,0\right)  =0  & x \in (0,L),\\
		h z^2_t(x,\rho,t)+z^2_\rho(x,\rho,t)=0 & x \in (0,L), ~\rho \in (0,1), ~ t>0,\\
		z^2(x,0,t)=y(x,t) & x \in (0,L),~ t>0,\\
		z^2(x,\rho,0)=0 & x \in (0,L), ~\rho \in (0,1),
	\end{cases}
\end{equation}
where $z^1(x,\rho,t)=v(x,t-\rho h)$ and $z^2(x,\rho,t)=w(x,t-\rho h)$. Define $w(x,t)=v(x,t)+y(x,t)$ and $\overline{z}(x,\rho,t)=z^1(x,\rho,t)+z^2(x,\rho,t)$, we get
\begin{equation}
	\begin{cases}\label{V6}
		w_{t}(x,t)+w_{x}(x,t)+w_{xxx}(x,t)-w_{xxxxx}(x,t)\\ \quad  \quad  \quad  \quad  \quad  \quad  \quad  \quad  \quad  \quad \quad+a\left(x\right)w(x,t)+b(x)\overline{z}(1)=0 & x \in (0,L),~ t>0,\\
		w\left(  0,t\right)  =w\left(  L,t\right)  =w_{x}\left(  0,t\right)
		=w_{x}\left(  L,t\right)  =w_{xx}\left(  L,t\right)  =0 & t>0,\\
		w\left(  x,0\right)  =u(x,T_0)  & x \in (0,L),\\
		h \overline{z}_t(x,\rho,t)+\overline{z}_\rho(x,\rho,t)=0 & x \in (0,L), ~\rho \in (0,1), ~ t>0,\\
		\overline{z}(x,0,t)=w(x,t) & x \in (0,L),~ t>0,\\
		\overline{z}(x,\rho,0)=z(x,\rho,T_0) & x \in (0,L), ~\rho \in (0,1).
	\end{cases}
\end{equation}
Therefore, $w(x,t)=u(x,t+T_0)$ and $\overline{z}(x,\rho,t)=z(x,\rho,t+T_0)$.  Thanks to the fact that $E_v(0) \leq \xi E_u(T_0)$ it follows that
\begin{align*}
	E_u(2T_0)
	\leq & \int_{0}^{L}v^2(x,T_0)dx+\int_{0}^{L}y^2(x,T_0)dx+h\int_{0}^{L}b(x)v^2(x,T_0-\rho h)d\rho dx\\
	&+h\int_{0}^{L}\int_{0}^{1}b(x)y^2(x,T_0-\rho h)d\rho dx\\
	\leq & 2E_v(T_0)+  2 \xi^2 \|b\|_{\infty}^{2} e^{(3\xi+1)T_0} \kappa E_{v}(0)\\
	\leq& \frac{\eta}{\xi} E_v(0)+\varepsilon E_u(T_0)\\
	\leq& \eta E_u(T_0)+\varepsilon E_u(T_0)\\
	\leq & (\eta +\varepsilon)^2E_u(0).
\end{align*}
Preceding in an analogous way, we get
\begin{equation*}
	E_u(mT_0) \leq (\eta +\varepsilon)^m E_u(0),
\end{equation*}
for all $m \in \mathbb{N}^*$. Now, to finish, let $t>T_0$, then there exists $m \in \mathbb{N}^*$ such that $t=mT_0+s$ with $0 \leq s<T_0$,  we have 
\begin{align*}
	E_u(t) 
	\leq& e^{(2\|b\|_\infty +\nu)s}e^{-\nu t}E_u(0),
\end{align*}
where 
\begin{equation}\label{nu}
	\nu=\frac{1}{T_0}\ln \left(\frac{1}{(\eta+\varepsilon)} \right),
\end{equation}
showing the proposition.
\end{proof}

\subsection{Proof of Theorem \ref{T1.2J}}
%


By a classical way (see Section \ref{ss2.3}) we can ensures that  the the system \eqref{fd1} is well-posed.  
Additionally, $u$ satisfies
$$
	\|(u(\cdot,t),u(\cdot,t-h(\cdot)))\|_{\mathcal{H}}^{2} \leq e^{2\xi\|b\|_\infty t} \|(u_0,z_0(\cdot,-h(\cdot)))\|_{\mathcal{H}}^{2},
$$
which implies that
$$\|u\|_{C([0,T],L^2(0,L))} \leq e^{\xi \|b\|_\infty T} \|(u_0,z_0(\cdot,-h(\cdot)))\|_{\mathcal{H}}
$$
and
$$\|u\|_{L^2(0,T,L^2(0,L))} \leq T^{\frac{1}{2}}e^{\xi \|b\|_\infty T} \|(u_0,z_0(\cdot,-h(\cdot)))\|_{\mathcal{H}}.$$

Let us now divide the rest of the proof in several steps.

\vspace{0.2cm}
\noindent\textbf{Step 1:} \textit{First estimate for the linear system associated to \eqref{fd1}.}
\vspace{0.2cm}

Multiplying the linear system associated to \eqref{fd1} by $u$, integrating by parts we have
\begin{align*}
	\frac{3}{2}\int_{0}^{T}\int_{0}^{L} u_x^2dx+\frac{5}{2}\int_{0}^{T}\int_{0}^{L} u_{xx}^{2}dx 
	\leq & \frac{L}{2}\|u_0\|_{L^2(0,L)}^2+\frac{1}{2}(1+2L\|a\|_{\infty}+L\|b\|_\infty)\int_{0}^{T}\int_{0}^{L}u^2(x,t)dxdt\\
	&+\frac{L}{2h}h\xi \|b\|_\infty\int_{0}^{T}\int_{0}^{L}u^2(x,t-h)d\rho dx \\
	\leq & \frac{L}{2}\|u_0\|_{L^2(0,L)}^2\\&+\frac{1}{2}\left(1+\frac{L}{h}+2L\|a\|_{\infty}+L\|b\|_\infty\right)Te^{2\xi\|b\|_\infty T} \|(u_0,z_0(\cdot,-h(\cdot)))\|_{\mathcal{H}}^{2}.
\end{align*}

\vspace{0.2cm}
\noindent\textbf{Step 2:} \textit{First estimate for the nonlinear system \eqref{fd1}.}
\vspace{0.2cm}

Now,  multiplying the nonlinear system \eqref{fd1} by $u$, integrating by parts we have 
\begin{align*}
	\frac{1}{2}\int_{0}^{T}\int_{0}^{L} u_x^2dx+\frac{5}{2}\int_{0}^{T}\int_{0}^{L} u_{xx}^{2}dx \leq & C_1(a,b,h,L)(1+Te^{2\xi\|b\|_\infty T}+e^{4\xi\|b\|_\infty T}) \|(u_0,z_0(\cdot,-h(\cdot)))\|_{\mathcal{H}}^{2}\\
	&+\frac{\sqrt{LT}}{4\varepsilon}\|(u_0,z_0(\cdot,-h(\cdot)))\|_{\mathcal{H}}^{4},\\
\end{align*}
where
\begin{equation*}
C_1(a,b,h,L)=	\frac{1}{2}\left(1+L+\frac{L}{h}+2L\|a\|_{\infty}+L\|b\|_\infty\right).
\end{equation*}
Here, we have used that as $H^1(0,L) \hookrightarrow C([0,L])$ we obtain, using the Cauchy-Schwarz inequality, that
\begin{align*}
	\int_{0}^{T} \int_{0}^{L}|u(x,t)|^3dxdt 
	\leq & \sqrt{LT}\|u\|_{L^\infty(0,T,L^2(0,L))}^{2}\|u\|_{L^2(0,T,H^1(0,L))}\\
	\leq & \sqrt{LT} e^{2\|b\|_\infty T} \|(u_0,z_0(\cdot,-h(\cdot)))\|_{\mathcal{H}}^{2} \|u\|_{L^2(0,T,H^1(0,L))}.
\end{align*}
Thus,
\begin{align*}
	\frac{3}{2}\int_{0}^{T}\int_{0}^{L} u_x^2dx+\frac{5}{2}\int_{0}^{T}\int_{0}^{L} u_{xx}^{2}dx \leq & C_2(a,b,h,L)(1+\sqrt{T}+Te^{2\xi\|b\|_\infty T}+e^{4\xi\|b\|_\infty T}) \\
	& \times (\|(u_0,z_0(\cdot,-h(\cdot)))\|_{\mathcal{H}}^{2}+\|(u_0,z_0(\cdot,-h(\cdot)))\|_{\mathcal{H}}^{4}),
\end{align*}
with
\begin{equation*}
	C_2(a,b,h,L)=	\frac{1}{2}\left(1+\frac{\sqrt{L}}{4\varepsilon}+L+\frac{L}{h}+2L\|a\|_{\infty}+L\|b\|_\infty\right).
\end{equation*}

\vspace{0.2cm}
\noindent\textbf{Step 3:} \textit{Second estimate for the linear system associated to \eqref{fd1}.}
\vspace{0.2cm}

Multiplying the linear system associated to \eqref{fd1} by $xu$, integrating by parts we also have,
\begin{align*}
	\frac{3}{2}\int_{0}^{T}\int_{0}^{L} u_x^2dx+\frac{5}{2}\int_{0}^{T}\int_{0}^{L} u_{xx}^{2}dx 
	\leq & LE_u(0)+\frac{1}{2}(1+2L\|a\|_{\infty}+2L\|b\|_\infty)\int_{0}^{T}\int_{0}^{L}u^2(x,t)dxdt\\
	&+\frac{Lh}{2}\int_{0}^{L}\int_{0}^{1}b(x)z_0^2(x,-\rho h)d\rho dx\\
	\leq & (2L+1+2L\|a\|_{\infty}+2L\|b\|_\infty)(1+Te^{2\|b\|_\infty T})E_u(0).
\end{align*}
	
\vspace{0.2cm}
\noindent\textbf{Step 4:} \textit{Second estimate for the nonlinear system \eqref{fd1}.}
\vspace{0.2cm}

Multiplying the system \eqref{fd1} by $xu$, integrating by parts  and using the fact that   $E_u(0) \leq 1$, yields that
\begin{align*}
	\frac{1}{2}\int_{0}^{T}\int_{0}^{L} u_x^2dx+\frac{5}{2}\int_{0}^{T}\int_{0}^{L} u_{xx}^{2}dx \leq & (2L+1+2L\|a\|_{\infty}+2L\|b\|_\infty)(1+Te^{2\|b\|_\infty T})E_u(0)\\
	& +\frac{\sqrt{LT}}{4\varepsilon}e^{4\|b\|_\infty T }E_u(0)
\end{align*}
which implies that
\begin{equation*}
		\frac{1}{2}\int_{0}^{T}\int_{0}^{L} u_x^2dx+\frac{5}{2}\int_{0}^{T}\int_{0}^{L} u_{xx}^{2}dx \leq C_3(a,b,h,L)(1+\sqrt{T}e^{4\|b\|_\infty T}+Te^{2\|b\|_\infty T})E_u(0)
\end{equation*}
where
\begin{equation*}
	C_3(a,b,h,L)=\left(2L+\frac{\sqrt{L}}{4\varepsilon}+1+2L\|a\|_{\infty}+2L\|b\|_\infty \right).
\end{equation*}
Here, we used again that $H^1(0,L) \hookrightarrow C([0,L])$ and so
\begin{align*}
	\int_{0}^{T} \int_{0}^{L}|u(x,t)|^3dxdt
	\leq & \sqrt{LT} e^{2\|b\|_\infty T} E_u(0)\|u\|_{L^2(0,T,H^1(0,L))}.
\end{align*}
Consequently,
\begin{equation*}
	\|u\|_{\mathcal{B}}^{2} \leq C_3(a,b,h,L)(1+\sqrt{T}e^{4\|b\|_\infty T}+e^{2\|b\|_{\infty}T}+2Te^{2\|b\|_\infty T})E_u(0)
\end{equation*}
where
\begin{equation*}
	C_3(a,b,h,L)=\left(2L+\frac{\sqrt{L}}{4\varepsilon}+1+2L\|a\|_{\infty}+2L\|b\|_\infty \right).
\end{equation*}

\vspace{0.2cm}
\noindent\textbf{Step 5:} \textit{Asymptotic behavior of the energy \eqref{Eb}.}
\vspace{0.2cm}

Pick the initial data $\|(u_0,z_0(\cdot,-h(\cdot)))\|_{\mathcal{H}} \leq r$, where $r$ to be chosen later. The solution $u$ of \eqref{fd1} can be written as $u=u^1+u^2$ where $u^1$ is solution of 
\begin{equation*}
	\begin{cases}
		u^1_{t}(x,t)+u^1_{x}(x,t)+u^1_{xxx}(x,t)-u^1_{xxxxx}(x,t)+a\left(x\right)u^1(x,t)\\ \quad  \quad  \quad  \quad  \quad  \quad  \quad  \quad  \quad  \quad \quad+b(x)u^1(x,t-h)=0 & x \in (0,L),~ t>0,\\
		u^1\left(  0,t\right)  =u^1\left(  L,t\right)  =u^1_{x}\left(  0,t\right)
		=u^1_{x}\left(  L,t\right)  =u^1_{xx}\left(  L,t\right)  =0 & t>0,\\
		u^1\left(  x,0\right)  =u_{0}\left(  x\right)  & x \in (0,L),\\
		u^1(x,t)=z_0(x,t) & x \in (0,L),~t \in (-h,0)
	\end{cases}
\end{equation*}
and $u^2$ is solution of 
\begin{equation*}
	\begin{cases}
		u^2_{t}(x,t)+u^2_{x}(x,t)+u^2_{xxx}(x,t)-u^2_{xxxxx}(x,t)+a\left(x\right)u^2(x,t)\\ \quad  \quad  \quad  \quad  \quad  \quad  \quad  \quad  \quad  \quad \quad+b(x)u^2(x,t-h)=-u(x,t)u_x(x,t) & x \in (0,L),~ t>0,\\
		u^2\left(  0,t\right)  =u^2\left(  L,t\right)  =u^2_{x}\left(  0,t\right)
		=u^2_{x}\left(  L,t\right)  =u^2_{xx}\left(  L,t\right)  =0 & t>0,\\
		u^2\left(  x,0\right)  =  0& x \in (0,L),\\
		u^2(x,t)=0& x \in (0,L),~t \in (-h,0).
	\end{cases}
\end{equation*}
Fix $\eta \in (0,1)$,  thanks to the Proposition \ref{P8J},  there exists $T_1>0$  such that $$e^{(2\|b\|_\infty+\nu)s- \nu T_1}< \frac{\eta}{2} \iff T_1>-\frac{1}{\nu}\ln\left(\frac{\eta}{2}\right)+\left(\frac{2\|b\|_\infty}{\nu}+1\right)s$$
with $\nu$ defined by \eqref{nu} satisfying $$E_{u^1}(T_1) \leq \frac{\eta}{2}E_{u^1}(0).$$ Thus, we have thanks to the previous inequality that
\begin{equation}\label{RRR}
	\begin{split}
	E_{u}(T_1)
	\leq & \int_{0}^{L}|u^1(x,T_1)|^2dx+\int_{0}^{L}|u^2(x,T_1)|^2dx\\
	&+h\int_{0}^{L}\int_{0}^{1}b(x)|u^1(x,T_1-\rho h)|^2d\rho dx+h\int_{0}^{L}\int_{0}^{1}b(x)|u^2(x,T_1-\rho h)|^2d\rho dx	\\
	\leq & 2E_{u^1}(T_1)+\int_{0}^{L}|u^2(x,T_1)|^2dx+h\|b\|_\infty \int_{0}^{L}\int_{0}^{1}|u^2(x,T_1-\rho h)|^2d\rho dx\\
\leq & \eta E_{u}(0)+\|(u^2(T_1,u^2(\cdot,T_1-h(\cdot))))\|_{\mathcal{H}}^{2}\\
	\end{split}
\end{equation}
So, with \eqref{RRR} in hand together with the estimates of the steps 1, 2, 3 and 4,  we get
\begin{align*}
	E_{u}(T_1)
	\leq&  \eta E_{u}(0)+e^{(3\xi+1)T_1}\|uu_x\|_{L^1(0,T_1,L^2(0,L)}^{2}\\
	\leq&  \eta E_{u}(0)+e^{(3\xi+1)T_1}2T_1^{\frac{1}{2}}\|u\|_{\mathcal{B}}^{4}\\
		\leq & E_{u}(0) (\eta + e^{(3\xi+1)T_1}2T_1^{\frac{1}{2}}C_3^2(a,b,h,L)(1+\sqrt{T_1}e^{4\|b\|_\infty T_1}+e^{2\|b\|_{\infty}T_1}+2T_1e^{2\|b\|_\infty T_1})^2r).
\end{align*}
Therefore, given $\varepsilon>0$ such that $\eta+\varepsilon<1$, we can take $r>0$ small enough such that
\begin{equation*}
r<\frac{\varepsilon}{	e^{(3\xi+1)T_1}2T_1^{\frac{1}{2}}C_3^2(a,b,h,L)(1+\sqrt{T_1}e^{4\|b\|_\infty T_1}+e^{2\|b\|_{\infty}T_1}+2T_1e^{2\|b\|_\infty T_1})^2},
\end{equation*}
in order to have
\begin{equation}\label{EE_U}
	E_u(T_1) \leq (\eta+\varepsilon)E_u(0),
\end{equation}
with $\eta+\varepsilon<1$. 

Finally, the solution of the problem
\begin{equation*}
	\begin{cases}
		v_{t}(x,t)+v_{x}(x,t)+v_{xxx}(x,t)-v_{xxxxx}(x,t)+a\left(x\right)v(x,t)\\ \quad  \quad  \quad  \quad  \quad  \quad  \quad  \quad  \quad  \quad \quad+b(x)v(x,t-h)+v(x,t)v_x(x,t)=0 & x \in (0,L),~ t>0,\\
		v\left(  0,t\right)  =v\left(  L,t\right)  =v_{x}\left(  0,t\right)
		=v_{x}\left(  L,t\right)  =v_{xx}\left(  L,t\right)  =0 & t>0,\\
		v\left(  x,0\right)  =  u(x,T_1)& x \in (0,L),\\
		v(x,t)=u(x,T_1+t)& x \in (0,L),~t \in (-h,0),
	\end{cases}
\end{equation*}
can be write as $u_1+u_2$, where $u^1$ is solution of 
\begin{equation*}
	\begin{cases}
		u^1_{t}(x,t)+u^1_{x}(x,t)+u^1_{xxx}(x,t)-u^1_{xxxxx}(x,t)\\ \quad  \quad  \quad  \quad  \quad  \quad  \quad  \quad  \quad  \quad \quad+a\left(x\right)u^1(x,t)+b(x)u^1(x,t-h)=0 & x \in (0,L),~ t>0,\\
		u^1\left(  0,t\right)  =u^1\left(  L,t\right)  =u^1_{x}\left(  0,t\right)
		=u^1_{x}\left(  L,t\right)  =u^1_{xx}\left(  L,t\right)  =0 & t>0,\\
		u^1\left(  x,0\right)  =u\left(  x, T_1\right)  & x \in (0,L),\\
		u^1(x,t)=u(x,T_1+t) & x \in (0,L),~t \in (-h,0),
	\end{cases}
\end{equation*}
and $u^2$ is solution of 
\begin{equation*}
	\begin{cases}
		u^2_{t}(x,t)+u^2_{x}(x,t)+u^2_{xxx}(x,t)-u^2_{xxxxx}(x,t)+a\left(x\right)u^2(x,t)\\ \quad  \quad  \quad  \quad  \quad  \quad  \quad  \quad  \quad  \quad \quad+b(x)u^2(x,t-h)=-v(x,t)v_x(x,t) & x \in (0,L),~ t>0,\\
		u^2\left(  0,t\right)  =u^2\left(  L,t\right)  =u^2_{x}\left(  0,t\right)
		=u^2_{x}\left(  L,t\right)  =u^2_{xx}\left(  L,t\right)  =0 & t>0,\\
		u^2\left(  x,0\right)  =  0& x \in (0,L),\\
		u^2(x,t)=0& x \in (0,L), ~t \in (-h,0),
	\end{cases}
\end{equation*}
From the uniqueness of solutions, we obtain that $v(x,t)=u(x,T_1+t)$ and $v(x,t-\rho h)=u(x,t+T_1-\rho h)$ with $\rho \in (0,1)$.  Moreover,  analogously as we did before 
\begin{align*}
	E_{u}(2T_1)
	\leq & \eta E_{u^1}(0)+\|(u^2(T_1,u^2(\cdot,T_1-h(\cdot))))\|_{\mathcal{H}}^{2}
\end{align*}

So, by the previous inequality, using again steps 1, 2, 3, 4 and \eqref{EE_U}, we have that
\begin{align*}
	E_{u}(2T_1)
	\leq&  \eta E_{u}(T_1)+e^{(3\xi+1)T_1}\|vv_x\|_{L^1(0,T_1,L^2(0,L)}^{2}\\
	\leq&  \eta E_{u}(T_1)+e^{(3\xi+1)T_1}2T_1^{\frac{1}{2}}\|v\|_{\mathcal{B}}^{4}\\
	\leq&  \eta E_{u}(T_1)+e^{(3\xi+1)T_1}2T_1^{\frac{1}{2}}C_3^2(a,b,h,L)(1+\sqrt{T_1}e^{4\|b\|_\infty T_1}+e^{2\|b\|_{\infty}T_1}+2T_1e^{2\|b\|_\infty T_1})^2E_v^2(0)\\
	\leq&   (\eta+\varepsilon)E_{u}(0)(\eta+e^{(3\xi+1)T_1}2T_1^{\frac{1}{2}}C_3^2(a,b,h,L)(1+\sqrt{T_1}e^{4\|b\|_\infty T_1}+e^{2\|b\|_{\infty}T_1}+2T_1e^{2\|b\|_\infty T_1})r)\\
	\leq& (\eta +\varepsilon)^2 E_{u}(0).
\end{align*}
From now on the proof follows the same steps as was done in Proposition \ref{P8J},  so Theorem \ref{T1.2J} is achieved.  \qed

\section{Asymptotic behavior of $\mu_i-$ system}\label{Sec4}
Let us return to study the behavior of the solution of $\mu_i-$system.  The task of this section is to prove the exponential stability for the solution of \eqref{J1}.

\subsection{Proof of Theorem \ref{P6J}\label{ss3.1}} We prove the local stability result which is based on the appropriate choice of Lyapunov functional.  We start proving that the energy,  associated to the solutions of \eqref{J1},  when \eqref{cdelay} is verified, decays exponentially.  To do it let us  consider the following Lyapunov functional
\begin{equation}\label{J26}
V(t)=E(t)+\alpha V_1(t)+\beta V_2(t),
\end{equation}
where $\alpha$ and $\beta$ are positive constants that will be fixed small enough later on and $E(t)$ is the energy defined by \eqref{J6}. Here,  $V_1$ and $V_2$ are defined by
\begin{equation}\label{J27}
V_1(t)=\int_{0}^{L}xu^2(x,t)dx
\end{equation}
and
\begin{equation}\label{J28}
V_2(t)=\frac{\xi}{2}\int_{0}^{L}\int_{0}^{1}(1-\rho)a(x)u^2(x,t-\rho h) d \rho dx,
\end{equation}
respectively.  It is clear that the two functional $E$ and $V$ are equivalent in the sense that
\begin{equation}\label{J29}
E(t)\leq V(t) \leq \left(1+\max\left\{2\alpha L, \beta\right\}\right)E(t).
\end{equation}
%
Now,  let $u$ be a solution of \eqref{J1} with $(u_0,z_0(\cdot,-h(\cdot))) \in \mathscr{D}(\mathcal{A})$ satisfying $\|(u_0,z_0(\cdot,-h(\cdot)))\|_{\mathcal{H}} \leq r$.  Differentiating \eqref{J27}, using the equation \eqref{J1} and integrating by parts, we obtain that
\begin{equation}\label{J32}
\begin{split}
	\frac{d}{dt}V_1(t)
	=& \int_{0}^{L}u^2(x,t)dx-3\int_{0}^{L}u_{x}^{2}(x,t)dx \\
	&-2\int_{0}^{L}xu(x,t)a(x)\mu_1u(x,t)dx-2\int_{0}^{L}xa(x)\mu_2u(x,t-h)u(x,t)dx \\
	&+\frac{2}{3}\int_{0}^{L}u^3(x,t)dx-5\int_{0}^{L}u_{xx}^{2}(x,t)dx.
	\end{split}
\end{equation}
Moreover, differentiating \eqref{J28} and using integration by parts, we have
\begin{equation}\label{J33}
\begin{split}
	\frac{d}{dt}V_2(t)=& \ \xi \int_{0}^{L}\int_{0}^{1} (1-\rho)a(x)u(x,t-\rho h)u_t(x,t-\rho h)d\rho dx  \\
	=&\ \frac{\xi}{2h}\int_{0}^{L}a(x)u^2(x,t)dx-\frac{\xi}{2h}\int_{0}^{L}\int_{0}^{1}a(x)u^2(x,t-\rho h)d\rho dx,
	\end{split}
\end{equation}
since
\begin{equation*}
	2\int_{0}^{1} (1-\rho)u(x,t-\rho h)u_\rho(x,t-\rho h)d\rho = -u^2(x,t)+\int_{0}^{1}u^2(x,t-\rho h) d\rho.
\end{equation*}
An argument analogous to the one made in Proposition \ref{P5JL} yields that
\begin{equation}\label{J25}
\begin{split}
	E_u'(t)\leq& -\frac{1}{2}u_{xx}^2(0)+\left(-\mu_1+\frac{\xi}{2h}+\frac{\mu_2}{2}\right)\int_{0}^{L}a(x)u^2(x)dx\\&+\left(\frac{\mu_2}{2}-\frac{\xi}{2h}\right)\int_{0}^{L}a(x)u^2(x,t-h)dx,
	\end{split}
\end{equation}
and consequently,
\begin{equation}\label{J24}
	E_u'(t)\leq-C_0\left[u_{xx}^2(0)+\int_{0}^{L}a(x)u^2(x)dx+\int_{0}^{L}a(x)u^2(x,t-h)dx\right]
\end{equation}
where $C_0>0$ is given by
\begin{equation*}
	C_0=\min\left\{\frac{1}{2},\mu_1-\frac{\xi}{2h}-\frac{\mu_2}{2},-\frac{\mu_2}{2}+\frac{\xi}{2h}\right\}        
\end{equation*}
for all solutions of system \eqref{J1}.  Thus,  from \eqref{J25}, \eqref{J26}, \eqref{J32}, \eqref{J33} and Cauchy-Schwarz inequality, we have for any $\gamma>0$,
\begin{align*}
	V'(t)+2\gamma V(t) 
	\leq &-\frac{1}{2}u_{xx}^{2}(0)+\left(-\mu_1+\frac{\xi}{2h}+\frac{\mu_2}{2}+2\alpha L\mu_1+\alpha L\mu_2+\frac{\beta \xi}{2h} \right) \int_{0}^{L}a(x)u^2(x,t)dx\\
	&+\left(\frac{\mu_2}{2}-\frac{\xi}{2h}+\alpha L \mu_2-\frac{\beta \xi}{2h}+\gamma \xi+\gamma \xi \beta\right)\int_{0}^{L}\int_{0}^{1}a(x)u^2(x,t-\rho h)d\rho dx\\
	&+(\alpha +\gamma +2\gamma L \alpha)\int_{0}^{L}u^2(x,t)dx-3\alpha\int_{0}^{L}u_{x}^{2}(x,t)dx\\
	&+\frac{2\alpha}{3}\int_{0}^{L}u^3(x,t)dx-5\alpha\int_{0}^{L}u_{xx}^{2}(x,t)dx.
\end{align*}
Thanks to the Poincaré's inequality, we get
\begin{align*}
V'(t)+2\gamma V(t) \leq & -\frac{1}{2}u_{xx}^{2}(0)+\left(-\mu_1+\frac{\xi}{2h}+\frac{\mu_2}{2}+2\alpha L\mu_1+\alpha L\mu_2+\frac{\beta \xi}{2h} \right) \int_{0}^{L}a(x)u^2(x,t)dx\\
&+\left(\frac{\mu_2}{2}-\frac{\xi}{2h}+\alpha L \mu_2 \right) \int_{0}^{L}\int_{0}^{1}a(x)u^2(x,t-\rho h)d\rho dx\\
&+\left(\frac{L^2}{\pi^2}(\alpha+\gamma +2\gamma L\alpha)-3\alpha \right)\int_{0}^{L}u_{x}^{2}(x,t)dx+\frac{2\alpha}{3}\int_{0}^{L}u^3(x,t)dx\\
&-5\alpha\int_{0}^{L}u_{xx}^{2}(x,t)dx+\left(\gamma \xi \beta+\gamma \xi -\frac{\beta \xi}{2h}\right)\int_{0}^{L}\int_{0}^{1}a(x)u^2(x,t-\rho h)d\rho dx.
\end{align*}
Similar argument as in \eqref{uuu}, Cauchy-Schwarz inequality, \eqref{J24} and since $H_{0}^{1}(0,L) \hookrightarrow C([0,L])$,  yields that
\begin{align*}
	\int_{0}^{L}u^3(x,t)dx 
	\leq &L^{\frac{3}{2}}r \|u_x(\cdot,t)\|_{L^2(0,L)}^{2}.
\end{align*}
Therefore,
\begin{align*}
V'(t)+2\gamma V(t) \leq & \left(\frac{L^2}{\pi^2}(\gamma(1+2L\alpha)+\alpha)-3\alpha+\frac{2\alpha L^\frac{3}{2}r}{3} \right)\int_{0}^{L}u_{x}^{2}(x,t)dx\\
&+\left(\gamma \xi \beta+\gamma \xi -\frac{\beta \xi}{2h}\right)\int_{0}^{L}\int_{0}^{1}a(x)u^2(x,t-\rho h)d\rho dx.
\end{align*}
Consequently, taking $\alpha, \beta, \gamma$ and $r$  as in the statement of proposition we have that
\begin{equation}\label{Vn}
	V'(t)+2\gamma V(t) \leq 0.
\end{equation}
Finally, from \eqref{J29} and \eqref{Vn}, we obtain
\begin{equation*}
	E(t) \leq V(t) \leq e^{-2\gamma t }V(0) \leq \left(1+\max\left\{2\alpha L, \beta\right\}\right)e^{-2\gamma t}E(0), \hbox{ for all } t>0.
\end{equation*}
By the density of $\mathscr{D}(\mathcal{A})$ in $\mathcal{H}$ the result extend to arbitraty $(u_0,z_{0}(\cdot,-h(\cdot))) \in \mathcal{H}$. \qed

\subsection{Proof of Theorem \ref{T1.1J}}


Now let us remove the hypotheses of the initial data being small in Theorem \eqref{P6J}.  To do it, let $u$ be the solution of \eqref{J1} with $(u_0,z_0(\cdot,-h(\cdot))) \in \mathscr{D}(\mathcal{A})$. Integrating \eqref{J24} between $0$ and $T>h$, we have
\begin{equation*}
	E(T)-E(0) \leq -C_0 \left(\int_{0}^{T} u_{xx}^2(0,t)dt+\int_{0}^{T}\int_{0}^{L}a(x)u^2(x)dxdt+\int_{0}^{T}\int_{0}^{L}a(x)u^2(x,t-h)dxdt \right),
\end{equation*}
where \begin{equation*}
	C_0=\min\left\{\frac{1}{2},\mu_1-\frac{\xi}{2h}-\frac{\mu_2}{2},-\frac{\mu_2}{2}+\frac{\xi}{2h}\right\},
\end{equation*}
which is equivalent to
\begin{equation}\label{J34}
\int_{0}^{T} u_{xx}^2(0,t)dt+\int_{0}^{T}\int_{0}^{L}a(x)u^2(x)dxdt+\int_{0}^{T}\int_{0}^{L}a(x)u^2(x,t-h)dxdt \leq \frac{1}{C_0}(E(0)-E(T)).
\end{equation}

Observe that the prove of Theorem \ref{T1.1J} is a direct consequence of the following observability inequality
\begin{equation}\label{J35}
	E(0) \leq C \left(\int_{0}^{T} u_{xx}^2(0,t)dt+\int_{0}^{T}\int_{0}^{L}a(x)u^2(x)dxdt+\int_{0}^{T}\int_{0}^{L}a(x)u^2(x,t-h)dxdt \right),
\end{equation}
for the solutions of the nonlinear system \eqref{J1}.  

In fact,  suppose that \eqref{J35} is verified and, as the energy is non-increasing, we have, thanks to \eqref{J34}, that
\begin{align*}
	E(T)
	\leq & C \left(\int_{0}^{T} u_{xx}^2(0,t)dt+\int_{0}^{T}\int_{0}^{L}a(x)u^2(x)dxdt+\int_{0}^{T}\int_{0}^{L}a(x)u^2(x,t-h)dxdt \right) \\
	\leq & \frac{C}{C_0}(E(0)-E(T)),
\end{align*}
which implies that 
\begin{equation}\label{J36}
E(T) \leq \gamma E(0), \hbox{ with } \gamma=\frac{\frac{C}{C_0}}{1+\frac{C}{C_0}}<1.
\end{equation}
The same argument used on the interval $[(m-1)T,mT]$ for $m=1,2,\dots$,yields that
\begin{equation*}
	E(mT) \leq \gamma E((m-1)T) \leq \cdots \leq \gamma^m E(0).
\end{equation*}
Thus, we have $$E(mT) \leq e^{-\nu m T}E(0)$$  with 
\begin{equation}\label{nu_a}
\nu=\frac{1}{T}\ln\left(1+\frac{C_0}{C}\right)>0.
\end{equation} 
For an arbitrary positive $t$, there exists $m \in \mathbb{N}^*$ such that $(m-1)T <t\leq mT$, and by the non-increasing property of the energy, we conclude that 
\begin{equation*}
	E(t) \leq E((m-1)T) \leq e^{-\nu(m-1)T}E(0) \leq \frac{1}{\gamma}e^{-\nu t}E(0).
\end{equation*}
By the density of $\mathscr{D}(\mathcal{A})$ in $\mathcal{H}$, we deduce that the exponential decay of the energy $E$ holds for any initial data in $\mathcal{H}$, showing so Theorem \ref{T1.1J}. \qed

\vspace{0.2cm}

Let us now prove the inequality \eqref{J35}.

\begin{proof}[Proof of the observalibity inequality]
First, we can obtain, similarly to \eqref{J13}, the following inequality
\begin{align}\label{J37}
	T\int_{0}^{L} u_0^2(x)dx \leq& \|u\|_{L^2(0,T,L^2(0,L))}^{2}+T \int_{0}^{T}u_{xx}^{2}(0,t)dt+T(2\mu_1+\mu_2)\int_{0}^{T}\int_{0}^{L}a(x)\mu_1u^2(x,t)dxdt \nonumber \\
	&+T\mu_2\int_{0}^{T}\int_{0}^{L}a(x)\mu_2u^2(x,t-h)dxdt
\end{align}
Now, multiplying $\eqref{J9}_4$ by $\xi a(x)z(x,\rho,s)$ and integrating in $(0,L) \times (0,1)$ we have that
\begin{equation}\label{d1}
	\frac{d}{ds}\frac{\xi }{2} \int_{0}^{L}a(x)\int_{0}^{1}(z(x,\rho,s))^2 \, d \rho dx=\frac{\xi}{2h}\int_{0}^{L} a(x)[(z(x,0,s))^2  -(z(x,1,s))^2] \,dx.
\end{equation}
Therefore,
\begin{equation}\label{d2}
\frac{\xi}{2h}\int_{0}^{t}\int_{0}^{L}a(x)(z^2(x,0,s)-z^2(x,1,s))dxds=\frac{\xi}{2}\int_{0}^{L}\int_{0}^{1}a(x)(z^2(x,\rho,t)-z^2(x,\rho,0))d\rho dx.
\end{equation}
From \eqref{d2} we obtain,
\begin{equation}\label{d3}
\begin{split}
\frac{\xi}{2}\int_{0}^{L}\int_{0}^{1}a(x)z^2(x,\rho,0)d\rho dx \leq& \frac{\xi}{2}\int_{0}^{L}\int_{0}^{1}a(x)z^2(x,\rho,t)d\rho dx\\&+\frac{\xi}{2h}\int_{0}^{t}\int_{0}^{L}a(x)z^2(x,1,s)dxds.
\end{split}
\end{equation}
So, integrating \eqref{d3} from $0$ to $T$ yields that
\begin{equation}\label{d4}
\begin{split}
T \frac{\xi}{2}\int_{0}^{L}\int_{0}^{1}a(x)z^2(x,\rho,0)d\rho dx \leq &\frac{\xi}{2}\int_{0}^{T}\int_{0}^{L}\int_{0}^{1}a(x)z^2(x,\rho,t)d\rho dxdt\\&+\frac{T\xi}{2h}\int_{0}^{T}\int_{0}^{L}a(x)z^2(x,1,t)dxdt.
\end{split}
\end{equation}
Noting that $z(x,\rho,t)= u(x,t-\rho h)$ it follows that
\begin{equation}
	\begin{aligned}\label{d5}
		\frac{\xi }{2}\int_{0}^{L}a(x)\int_{0}^{1}(z(x,\rho,0))^2 \, d \rho dx ={}& \frac{\xi }{2}\int_{0}^{L}a(x)\int_{0}^{1}( u(x,-\rho h))^2 \, d \rho dx\\
		= {}& \frac{\xi }{2}\int_{0}^{L}a(x)\int_{0}^{-h}( u(x,s))^2 \,\left(-\frac{1}{h}\right) ds dx\\
		\leq{}& \frac{\xi }{2 h}\int_{0}^{L}a(x)\int_{0}^{T}(z(x,1,t))^2 \, dt dx,
	\end{aligned}
\end{equation}
where, in the second equality, we have used the following change of variable $s=-\rho h$. 
From \eqref{d1} and \eqref{d5} we also have
\begin{equation}\label{d6}
\begin{split}
	\frac{\xi }{2} \int_{0}^{L}a(x)\int_{0}^{1}(z(x,\rho,t))^2 \, d \rho dx
	\leq&\  \frac{\xi }{2 h}\int_{0}^{L}a(x)\int_{0}^{T}(z(x,1,t))^2 \, dt dx\\&+\frac{\xi}{2h}\int_{0}^{T} \int_{0}^{L} a(x)(z(x,0,t))^2\,dxdt,
\end{split}
\end{equation}
Hence, from \eqref{d4} and \eqref{d6} we obtain
\begin{equation}\label{d7}
\begin{split}
T \frac{\xi}{2}\int_{0}^{L}\int_{0}^{1}a(x)z^2(x,\rho,0)d\rho dx 
\leq & \left(\frac{\xi }{2 h}+\frac{T\xi }{2 h} \right) \int_{0}^{T}\int_{0}^{L}a(x)u^2(x,t-h)dxdt\\&+\frac{\xi }{2 h}\int_{0}^{T}\int_{0}^{L}a(x)u^2(x,t)dxdt.
\end{split}
\end{equation}

Gathering \eqref{d7} with \eqref{J37}, we see that in order to prove the observability inequality \eqref{J35} it is sufficient to prove that for any $T,R>0$ there exists $K:=K(R,T)>0$ such that
\begin{equation}\label{J38}
\begin{split}
\|u\|_{L^2(0,T,L^2(0,L))}^{2} \leq K \left(\int_{0}^{T} u_{xx}^2(0,t)dt+\int_{0}^{T}\int_{0}^{L}a(x)u^2(x)dxdt+\int_{0}^{T}\int_{0}^{L}a(x)u^2(x,t-h)dxdt \right)
\end{split}
\end{equation}
holds for all solutions of the nonlinear system \eqref{J1} with $\|(u_0,z_0(\cdot,-h(\cdot)))\|_{\mathcal{H}} \leq R$.  

Let us now argue by contradiction.  If \eqref{J38} does not hold, there exists a sequence $\{u^n\}_{n\in\mathbb{N}} \subset \mathcal{B}$ of solutions to system \eqref{J1} with $\|(u_0^n,z_0^n(\cdot,-h(\cdot)))\|_{\mathcal{H}} \leq R$ such that

\begin{equation*}
	\lim_{n \rightarrow \infty} \frac{\|u^n\|_{L^2(0,T,L^2(0,L))}^{2}}{\|u_{xx}^{n}(0,\cdot)\|_{L^2(0,T)}^{2}+\int_{0}^{T}\int_{0}^{L}a(x)|u^n(x,t)|^2dxdt+\int_{0}^{T}\int_{0}^{L}a(x)|u^n(x,t-h)|^2dxdt}=\infty.
\end{equation*}
We define $\lambda_n=\|u^n\|_{L^2(0,T,L^2(0,L))}$ and $v^{n} =\frac{u^n}{\lambda_n}$. Then, $v_n$ satisfies
\begin{equation}
	\begin{cases}\label{J39}
		v^{n}_{t}(x,t)+v^{n}_{x}(x,t)+v^{n}_{xxx}(x,t)-v^{n}_{xxxxx}(x,t)+\lambda_nv^{n}v^{n}_{x}(x,t)\\ \quad  \quad  \quad  \quad  \quad  \quad  \quad  \quad  \quad  \quad  +a\left(x\right)(\mu_1 v^{n}(x,t)+\mu_2v^{n}(x,t-h))=0 & x \in (0,L),~ t>0,\\
		v^{n}\left(  0,t\right)  =v^{n}\left(  L,t\right)  =v^{n}_{x}\left(  0,t\right)
		=v^{n}_{x}\left(  L,t\right)  =v^{n}_{xx}\left(  L,t\right)  =0 & t>0,\\
		v^{n}\left(  x,0\right)  =\frac{u^{n}_{0}}{\lambda_n}\left(  x\right)  & x \in (0,L),\\
		v^{n}(x,t)=\frac{z_{0}^{n}}{\lambda_n}(x,t) & x \in (0,L),~t \in (-h,0),
	\end{cases}
\end{equation}
\begin{equation}\label{J40}
\|v^{n}\|_{L^2(0,T,L^2(0,L))}=1
\end{equation}
and
\begin{equation}\label{J41}
\|v^{n}_{xx}(0,\cdot)\|_{L^2(0,T)}^{2}+\int_{0}^{T}\int_{0}^{L}a(x)\left(|v^n(x,t)|^2+|v^n(x,t-h)|^2\right)dxdt \rightarrow 0 \hbox{ as } n \rightarrow \infty.
\end{equation}

\vspace{0.2cm}
\noindent\textbf{Claim 1.} \textit{$\{v^n(\cdot,0)\}$ is bounded in $L^2(0,L)$.}
\vspace{0.2cm}

Indeed,  since
\begin{equation*}
	\int_{0}^{T}\int_{0}^{L}(T-t)(v^n)^2v_{x}^{2}dxdt=0.
\end{equation*}
we have, as for the linear case, that
\begin{equation}\label{J41.1}
\begin{split}
	\|v^n(x,0)\|_{L^2(0,L)}^{2} \leq & \frac{1}{T} \|v^n\|_{L^2(0,T,L^2(0,L))}^{2}+\|v_{xx}^{n}(0,\cdot)\|_{L^2(0,T)}^{2}\\
	&+ (2\mu_1+\mu_2)\int_{0}^{L}\int_{0}^{T}a(x)|v^n(x,t)|^2dxdt\\&+\int_{0}^{L}\int_{0}^{T}a(x)|v^n(x,t-h)|^2dxdt.
	\end{split}
\end{equation}
Gathering \eqref{J40}, \eqref{J41} and \eqref{J41.1} the Claim 1 follows.

\vspace{0.2cm}
\noindent\textbf{Claim 2.} \textit{$\{\sqrt{a(x)}v^n(\cdot,-h(\cdot))\}$ is bounded in $L^2((0,L)\times(0,1))$ and $\{\lambda_n\}$ is bounded in $\mathbb{R}$.}
\vspace{0.2cm}

In fact, as we have that
\begin{equation*}
\begin{split}
T \frac{\xi}{2}\int_{0}^{L}\int_{0}^{1}a(x)|z_{0}^{n}(x,\rho,0)|^2d\rho dx \leq& \left(\frac{\xi }{2 h}+\frac{T\xi }{2 h} \right) \int_{0}^{T}\int_{0}^{L}a(x)|u^n(x,t-h)|^2dxdt\\&+\frac{\xi }{2 h}\int_{0}^{T}\int_{0}^{L}a(x)|u^n(x,t)|dxdt	
	\end{split}
\end{equation*}
it follows that
\begin{equation*}
\begin{split}
T \frac{\xi}{2}\int_{0}^{L}\int_{0}^{1}a(x)\frac{1}{\lambda_n^2}|z_{0}^{n}(x,\rho,0)|^2d\rho dx \leq& \left(\frac{\xi }{2 h}+\frac{T\xi }{2 h} \right) \int_{0}^{T}\int_{0}^{L}a(x)|v^n(x,t-h)|^2dxdt\\&+\frac{\xi }{2 h}\int_{0}^{T}\int_{0}^{L}a(x)|v^n(x,t)|dxdt,	
	\end{split}
\end{equation*}
and consequently,  $\{\sqrt{a(x)}v^n(\cdot,-h(\cdot))\}$ in bounded.  Moreover, thanks to \eqref{J24} we see that
\begin{equation*}
	\lambda_{n}^{2}=\|u^n\|_{L^2(0,T,L^2(0,L))}^{2} \leq T \|(u_{0}^{n}(\cdot),z_{0}^{n}(\cdot,-h(\cdot)))\|_{\mathcal{H}}^{2} \leq TR^2,
\end{equation*}
that is,  $\{\lambda_n\}$ is bounded, and so, the Claim 2 holds.

\vspace{0.2cm}
\noindent\textbf{Claim 3.} \textit{ $\{v^n\}$ is bounded in $L^2(0,T,H^2(0,L))$.}
\vspace{0.2cm}

This follows noting first that, as in the proof of Proposition \ref{P1J}, we have that
\begin{equation*}
\begin{split}
	\frac{3}{2}\int_{0}^{T}\int_{0}^{L} |v^{n}_{x}(x,t)|^2dx+\frac{5}{2}\int_{0}^{T}\int_{0}^{L}|v^{n}_{xx}(x,t)|^{2}dx &= \frac{1}{2}\int_{0}^{L}x((v_{0}^{n})^{2}(x)-(v^n)^2(x,T))dx\\ &+\frac{1}{2}\int_{0}^{T}\int_{0}^{L}|v^n(x,t)|^2 dxdt\\
	&-\int_{0}^{T}\int_{0}^{L}xa(x)\mu_1|v^n(x,t)|^2 dxdt \\
	&-\int_{0}^{T} \int_{0}^{L}xa(x)\mu_2v^n(x,t-h)v^n(x,t)dxdt\\
	&-\int_{0}^{T}\int_{0}^{L}x\lambda_n v^nv_{x}^{n}v^ndxdt.
	\end{split}
\end{equation*}
Now, observe that 
\begin{align*}
-\int_{0}^{T}\int_{0}^{L}x\lambda_n v^nv_{x}^{n}v^ndxdt 
\leq & \sqrt{LT} \lambda_n\|v^n\|_{L^\infty(0,T,L^2(0,L))}^{2}\|v^n\|_{L^2(0,T,H^1(0,L))}\\
\leq &\frac{\sqrt{LT}\lambda_{n} \xi}{2}\|v^n\|_{L^2(0,T,H^1(0,L))}\int_{0}^{L} |v_{0}^{n}(x)|^2dx\\
&+\frac{\sqrt{LT}\lambda_{n} \xi}{2}\|v^n\|_{L^2(0,T,H^1(0,L))}\int_{0}^{L}\int_{0}^{1}a(x)|v^n(x,-\rho h)|^2d\rho dx.
\end{align*}
Thus,
for $\varepsilon>0$ small enough, we have, putting the two previous inequalities together, that
\begin{equation*}
	\begin{split}
	\|v^n\|_{L^2(0,T,H^2(0,L))}^{2} \leq& L\|v_{0}^{n}\|_{L^2(0,L)}^{2}+\|v^n\|_{L^2(0,T,L^2(0,L))}^{2}\\
&+ L(2\mu_1+\mu_2)\int_{0}^{T}\int_{0}^{L}a(x)|v^n(x,t)|^2dxdt\\&+L\mu_2\int_{0}^{T}\int_{0}^{L}a(x)|v^n(x,t-h)|^2dxdt\\
&+\sqrt{LT}\lambda_{n} \left(\frac{\xi}{2}\int_{0}^{L} |v_{0}^{n}(x)|^2dx+\frac{\xi}{2}\int_{0}^{L}\int_{0}^{1}a(x)|v^n(x,-\rho h)|^2d\rho,
		 dx\right)^2\frac{1}{2\varepsilon}
		 	\end{split}	
\end{equation*}
showing the Claim 3. 

\vspace{0.2cm}
\noindent\textbf{Claim 4.} \textit{$\{v^nv_{x}^{n}\}$ is bounded in $L^2(0,T,L^1(0,L))$.}
\vspace{0.2cm}

This claim is a direct consequence of the following inequality
\begin{equation*}
	\|v^nv_{x}^{n}\|_{L^2(0,T,L^1(0,L))}\leq \|v^n\|_{C([0,T],L^2(0,L))}\|v^n\|_{L^2(0,T,H^2(0,L))},
\end{equation*}
where we used the Cauchy-Schwarz inequality. 

Thus, putting together all this results we showed that $$v_{t}^{n}(x,t)=-(v^{n}_{x}(x,t)+v^{n}_{xxx}(x,t)-v^{n}_{xxxxx}(x,t)+\lambda_nv^{n}v^{n}_{x}(x,t)+a\left(x\right)(\mu_1 v^{n}(x,t)+\mu_2v^{n}(x,t-h)))$$ is bounded in $L^2(0,T,H^{-3}(0,L))$ and using the classical compactness results (see e.g.  \cite{Simon}), we obtain that $\{v^n\}$ is relatively compact in $L^2(0,T,L^2(0,L))$. Thus, there exists a subsequence of $\{v^n\}$, still denoted by $\{v^n\}$,  such that 
$$v^n\longrightarrow v,  \quad \text{strongly in}\quad L^2(0,T,L^2(0,L)),$$
verifying $$\|v\|_{L^2(0,T,L^2(0,L))}=1.$$  Furthermore, by weak lower semicontinuity, we have 
\begin{equation*}
	v(x,t)=0 \in \omega \times (0,T) \hbox{ and } v_{xx}(0,t)=0 \hbox{ in } (0,T).
\end{equation*}
Since $\{\lambda_n\}$ is bounded, we can also extract a subsequence, still denoted by $\{\lambda_n\}$ which converges to $\lambda \geq 0$. Consequently, the limit $v$ satisfies
\begin{equation}\label{vlimit}
	\begin{cases}
		v_{t}(x,t)+v_{x}(x,t)+v_{xxx}(x,t)-v_{xxxxx}(x,t)+\lambda v(x,t)v_{x}(x,t)=0 & x \in (0,L),~ t>0,\\
		v\left(  0,t\right)  =v\left(  L,t\right)  =v_{x}\left(  0,t\right)
		=v_{x}\left(  L,t\right)  =v_{xx}\left(  L,t\right)  =0 & t>0,\\
		v(x,t)=0  & x \in \omega, t \in (0,T),\\
		v_{xx}(0,t)=0 & t \in(0,T),\\
		\|v\|_{L^2(0,T,L^2(0,L))}=1.
	\end{cases}
\end{equation}
At this moment we shall divide our proof into two cases: 

\smallskip
\noindent{Case (i):~$\lambda=0$.}
\smallskip

In this case, the system satisfied by $v$ is linear and we can apply the Holmgren's uniqueness theorem to obtain that $v=0$, which contradicts the fact that $\|v\|_{L^2(0,T,L^2(0,L))}=1$.

\smallskip
\noindent{Case (ii):~$\lambda>0$.}
\smallskip

For that case, we need to prove that $v \in L^2(0,T,H^5(0,L))$. In this way, let us consider $u=v_t$. Then, $u$ is solution of the following equation
\begin{equation*}
	\begin{cases}
		u_{t}(x,t)+u_{x}(x,t)+u_{xxx}(x,t)-v_{xxxxx}(x,t)\\ \quad  \quad  \quad  \quad  \quad  \quad  \quad  \quad  \quad  \quad +\lambda u(x,t)v_{x}(x,t)+\lambda v(x,t)u_{x}(x,t)=0 & x \in (0,L),~ t>0,\\
		u\left(  0,t\right)  =u\left(  L,t\right)  =u_{x}\left(  0,t\right)
		=u_{x}\left(  L,t\right)  =u_{xx}\left(  L,t\right)  =0 & t>0,\\
		u(x,t)=0  & x \in \omega, t \in (0,T),\\
		u_{xx}(0,t)=0 & t \in(0,T),\\
		u(x,0)=-v_x(x,0)-v_{xxx}(x,0)-v_{5x}(x,0)-\lambda v(x,0)v_x(x,0) \in H^{-5}(0,L)\\
	\end{cases}
\end{equation*}
%
Thus,  $u(\cdot,0) \in L^2(0,L)$ and so $u=v_t \in \mathcal{B}$. It follows from \eqref{vlimit}
that $u_{xxxxx}\in L^{2}\left( (0,L) \times (0,T)\right)  $. Therefore,
\[
u\in L^{2}(  0,T;H^{5}\left(  0,L\right)) \cap H^{1}(
0,T;H^{2}\left(  0,L\right) )
\]
which is sufficiently to the unique continuation principle from  \cite{Saut}  be applied. This gives $u\equiv0$ in $(0,L) \times (0,T)$ which completes the proof. 
\end{proof}

\appendix 

\section{Study of an auxiliary system}\label{ApA}
The goal of this appendix is to treat the system \eqref{J42} linearized around $0$.
\begin{equation}
	\begin{cases}\label{J44}
		u_{t}(x,t)+u_{x}(x,t)+u_{xxx}(x,t)-u_{xxxxx}(x,t)+a\left(x\right)u(x,t)\\ \quad  \quad  \quad  \quad  \quad  \quad  \quad  \quad  \quad  \quad \quad+b(x)u(x,t-h)+\xi b(x)u(x,t)=0 & x \in (0,L),~ t>0,\\
		u\left(  0,t\right)  =u\left(  L,t\right)  =u_{x}\left(  0,t\right)
		=u_{x}\left(  L,t\right)  =u_{xx}\left(  L,t\right)  =0 & t>0,\\
		u\left(  x,0\right)  =u_{0}\left(  x\right)  & x \in (0,L),\\
		u(x,t)=z_0(x,t) & x \in (0,L),~t \in (-h,0),
	\end{cases}
\end{equation}
The results contained herein are essential to prove one of the main results of this work.

\subsection{Well-posedness of the auxiliary system}
We start showing that system \eqref{J44} is well-posed. As in Section \ref{Sec3}, setting $z(x,\rho,t)=u(x,t-\rho h)$ for any $x \in (0,L)$, $\rho \in (0,1)$ and $t>0$, $(u(\cdot,t),z(\cdot,\cdot,t))$ satisfies the system
\begin{equation}
	\begin{cases}\label{J9.1}
		u_{t}(x,t)+u_{x}(x,t)+u_{xxx}(x,t)-u_{xxxxx}(x,t)+a\left(x\right)u(x,t)\\ \quad  \quad  \quad  \quad  \quad  \quad  \quad  \quad  \quad  \quad \quad+b(x)z(1)+\xi b(x)u(x,t)=0 & x \in (0,L),~ t>0,\\
		u\left(  0,t\right)  =u\left(  L,t\right)  =u_{x}\left(  0,t\right)
		=u_{x}\left(  L,t\right)  =u_{xx}\left(  L,t\right)  =0 & t>0,\\
		u\left(  x,0\right)  =u_{0}\left(  x\right)  & x \in (0,L),\\
		h z_t(x,\rho,t)+z_\rho(x,\rho,t)=0 & x \in (0,L), ~\rho \in (0,1), ~ t>0,\\
		z(x,0,t)=u(x,t) & x \in (0,L),~ t>0,\\
		z(x,\rho,0)=z_0(x,-\rho h) & x \in (0,L), ~\rho \in (0,1).
	\end{cases}
\end{equation}
Consider also the Hilbert space $\mathcal{H}=L^2(0,L) \times L^2((0,L)\times(0,1))$ with the inner product 
\begin{equation*}
	((u,z), (v,w))_{\mathcal{H}} = \int_0^L uv dx+h\xi \|b\|_\infty \int_{0}^{L}\int_0^1 z(x,\rho)w(x,\rho) dxd\rho.
\end{equation*}
Rewriting system \eqref{J9.1} as a first order system
\begin{equation}\label{J45}
	\begin{cases}
		\displaystyle \frac{\partial U}{\partial t}(t)=\mathcal{A}_0U(t)  \\
		U(0)={}  (u_{0}(x), z_0(x,-\rho h )),
	\end{cases}
\end{equation}
with the unbounded operator $\mathcal{A}_0:\mathscr{D}(\mathcal{A}) \subset \mathcal{H} \rightarrow \mathcal{H}$ given by
\begin{equation}\label{A0}
	\mathcal{A}_0(u,z)=(-u_x-u_{xxx}+u_{xxxxx}-a(x)u-\xi b(x)u -b(x) z(\cdot,1),-h^{-1}z_\rho)
\end{equation}
with domain
\begin{equation}\label{2.130}
	\mathscr{D}(\mathcal{A}_0)  = \left\{(u,z) \in \mathcal{H}: 
	\begin{array}{c}
		\displaystyle	u \in H^5(0,L), u(0)=u(L)=u_x(0)=u_x(L)=u_{xx}(L)=0,\\
		\displaystyle  z_\rho \in L^2((0,L)\times (0,1)),	z(0)=u,
	\end{array}
	\right\}
\end{equation}
so the following result holds.

\begin{theorem}\label{T4.2J}
	Assume that $a$ and $b$ are nonnegative functions in $L^\infty(0,L)$ with $b(x) \geq b_0>0$ in $\omega$, $U_0 \in \mathcal{H}$ and $\xi>1$. Then, there exists a unique mild solution $U \in C([0,\infty),\mathcal{H})$ for system \eqref{J45}. Moreover, if $U_0 \in \mathscr{D}(\mathcal{A}_0)$, then the solution is classical and satisfies
	\begin{equation*}
		U \in C([0,\infty),\mathscr{D}(\mathcal{A}_0)) \cap C^1([0,\infty),\mathcal{H}).
	\end{equation*}
\end{theorem}
\begin{proof}
Let $U=(u,z) \in \mathscr{D}(\mathcal{A}_0)$, then we have
\begin{align*}
	(\mathcal{A}_0 U,U)
%
\leq & \frac{(1+\xi)}{2} \|b\|_\infty \int_{0}^{L}\int_{0}^{L}u^2(x)dx.
\end{align*}
It is note difficult to prove that the adjoint of $\mathcal{A}_0$ denoted by $\mathcal{A}^*_0$ is defined by
	\begin{equation}\label{A0star}
	\mathcal{A}_{0}^{*}(u,z)=(u_x+u_{xxx}-u_{xxxxx}-a(x)u-\xi b(x)u+\xi\|b\|_\infty z(\cdot,0),h^{-1}z_\rho)
\end{equation}
with domain
\begin{equation}\label{2.140}
	\mathscr{D}(\mathcal{A}_{0}^{*})  = \left\{(u,z) \in \mathcal{H}: 
	\begin{array}{c} \displaystyle
		u \in H^5(0,L),  u(0)=u(L)=u_x(0)=u_x(L)=u_{xx}(0)=0,\\
		\displaystyle z_\rho \in L^2((0,L) \times (0,1)), z(x,1)=-\frac{b(x) }{\xi \|b\|_\infty }u(x)
	\end{array}
	\right\}. 
\end{equation}
Let $U=(u,z) \in \mathscr{D}(\mathcal{A}_{0}^{*})$, then
\begin{align*}
(\mathcal{A}_{0}^{*}U,U)_{\mathcal{H}} 
\leq & \frac{(1+\xi)}{2} \|b\|_\infty \int_{0}^{L}u^2(x)dx.
\end{align*}
Hence, for $\lambda=\frac{(1+\xi)}{2} \|b\|_\infty$,
\begin{equation*}
((\mathcal{A}_0-\lambda I)U,U)_{\mathcal{H}} \leq 0	 \hbox{ and } ((\mathcal{A}_0-\lambda I)^*V,V)_{\mathcal{H}} \leq 0,
\end{equation*}
for all $U \in \mathscr{D}(\mathcal{A}_0)$ and $V \in \mathscr{D}(\mathcal{A}_{0}^{*})$. Finally, since $\mathcal{A}_0-\lambda I$ is a densely defined closed linear operator, and both $\mathcal{A}_0-\lambda I$ and $(\mathcal{A}_0-\lambda I)^*$ are dissipative, then $\mathcal{A}_0$ is the infinitesimal generator of a $C_0$-semigroup on $\mathcal{H}$ (see for instance Corollary 4.4 and remark before Corollary 3.8 in \cite{Pazy}).
\end{proof}

\subsection{Exponential stability of the auxiliary system}
We denote by $\{e^{\mathcal{A}_0t}, t \geq 0\}$ the $C_0$-semigroup associated with $\mathcal{A}_0$. To prove the exponential stability of the system \eqref{J44}, we closely follow the Subsection \ref{ss3.1}. Precisely, we choose the following Lyapunov functional
\begin{equation}\label{J46}
V(t)=E(t)+\alpha V_1(t)+\beta V_2(t),
\end{equation}
where $\alpha$ and $\beta$ are positive constants that will be fixed small enough later on, $E$ is the energy defined by \eqref{J43}, $V_1$ is defined by \eqref{J27} and $V_2$ is defined by
\begin{equation}\label{J47}
V_2(t)=\frac{h}{2}\int_{0}^{L}\int_{0}^{1}(1-\rho)b(x)u^2(x,t-\rho h)d\rho dx.
\end{equation}
It is clear that the two energies $E$ and $V$ are equivalent, in the sense that
\begin{equation}\label{J48}
E(t)\leq V(t) \leq \left(1+\max\left\{2\alpha L, \frac{\beta}{\xi}\right\}\right)E(t).
\end{equation}
Thus, the following result gives a positive answer for the exponential stability to the system \eqref{J44}.

\begin{proposition}\label{P7J}
	Assume that $a$ and $b$ are nonnegative function in $L^\infty(0,L)$, $b(x)\geq b_0>0$ in $\omega$, $L<\pi \sqrt{3}$ and $\xi>1$. Then, for every $(u_0,z_0(\cdot,-h(\cdot))) \in \mathcal{H}$, the energy of system \eqref{J44}, denoted by $E$ and defined by \eqref{J43}, decays exponentially. More precisely, there exists two positive constants $\gamma$ and $\kappa$ such that
	\begin{equation*}
		E(t)\leq \kappa E(0) e^{-2\gamma t} \hbox{ for all }  t>0,
	\end{equation*}
	with 
	\begin{align*}
		\gamma \leq & \min\left\{\frac{(3\pi^2-L^2)\alpha}{L^2(1+2\alpha L)},\frac{\beta }{2h(\xi+\beta)} \right\}\\
		\kappa = & \left(1+\max\left\{2\alpha L,\frac{\beta}{\xi}\right\}\right),
	\end{align*}
	where $\alpha$ and $\beta$ are positive constants such that
	\begin{align*}
		\alpha<&\frac{\xi-1}{2L(1+2\xi)},\\
		\beta<& \xi-1-2\alpha L(1+2\xi).
	\end{align*}
\end{proposition}
\begin{proof}
Let $u$ be a solution of \eqref{J44} with $(u_0,z_0(\cdot,-h(\cdot))) \in \mathscr{D}(\mathcal{A}_0)$. Differentiating \eqref{J27} and using the first equation of \eqref{J44},  we have that
\begin{align*}
	\frac{d}{dt}V_1(t)=&-3\int_{0}^{L}u_{x}^{2}(x,t)dx+\int_{0}^{L}u^2(x,t)dx-5\int_{0}^{L}u_{xx}^{2}(x,t)dx\\
	&-2\int_{0}^{L}xa(x)u^2(x,t)dx-2\int_{0}^{L}xb(x)u(x,t)u(x,t-h)dx-2\int_{0}^{L}x\xi b(x)u^2(x,t)dx.
\end{align*}
Moreover, differentiating \eqref{J47}, using integration by parts, we obtain
\begin{align*}
	\frac{d}{dt}V_2(t)
	=& \frac{1}{2}\int_{0}^{L}b(x)u^2(x,t)dx-\frac{1}{2}\int_{0}^{L}\int_{0}^{1}b(x)u^2(x,t-\rho h)d\rho dx.
\end{align*}
Consequently, for any $\gamma>0$, we get
%
\begin{align*}
	\frac{d}{dt}V(t)+2 \gamma V(t)
	\leq & \frac{1}{2}\int_{0}^{L}b(x)(1-\xi+\beta+2\alpha L(1+2\xi))u^2(x,t)dx \\&+\frac{1}{2}\int_{0}^{L}b(x)(1-\xi + 2\alpha L)u^2(x,t-h)dx\\
	&+\left(\frac{L^2}{\pi^2}(\alpha + \gamma +2 \alpha \gamma L)-3\alpha \right)\int_{0}^{L}u_{x}^{2}(x,t)dx\\
	&+\int_{0}^{L}\int_{0}^{1}b(x)\left(\gamma \xi h+\gamma \beta h -\frac{\beta}{2}\right)u^2(x,t-h)d\rho dx.
\end{align*}
Therefore, for $\alpha, \beta$ and $\gamma$ chosen as in the statement of proposition we have
\begin{equation}\label{Vn0}
	V'(t)+2\gamma V(t) \leq 0.
\end{equation}
From \eqref{J48} and \eqref{Vn0}, we obtain
\begin{equation*}
	E(t) \leq V(t) \leq e^{-2\gamma t }V(0) \leq \left(1+\max\left\{2\alpha L,\frac{\beta}{\xi}\right\}\right)e^{-2\gamma t}E(0), \hbox{ for all } t>0.
\end{equation*}
By the density of $\mathscr{D}(\mathcal{A})$ in $\mathcal{H}$ the result extend to arbitraty $(u_0,z_{0}(\cdot,-h(\cdot))) \in \mathcal{H}$.
\end{proof}

\subsection*{Acknowledgments} Capistrano–Filho was supported by CNPq grant 408181/2018-4, CAPES- PRINT grant 88881.311964/2018-01, MATHAMSUD grants 88881.520205/2020-01, 21-MATH-03 and Propesqi (UFPE).  Gonzalez Martinez was supported by FACEP grant BFP-0065-1.01/21.  This work was done during the postdoctoral visit of the second author at the Universidade Federal de Pernambuco, who thanks the host institution for the warm hospitality.

\end{document}